\documentclass[11pt]{amsart}
\usepackage{amsmath}
\usepackage{epsfig}
\usepackage{graphicx}
\usepackage{eufrak}
\usepackage{dsfont}
\usepackage[english]{babel}
\usepackage{color}

\usepackage{palatino, mathpazo}
\usepackage{amscd}      
\usepackage{amssymb}
\usepackage{dsfont}
\usepackage{xypic}      
\LaTeXdiagrams          
\usepackage[all,v2]{xy}
\xyoption{2cell} \UseAllTwocells \xyoption{frame} \CompileMatrices
\usepackage{latexsym}
\usepackage{epsfig}
\usepackage{enumerate}

\usepackage{latexsym}
\usepackage{epsfig}
\usepackage{amsfonts}
\usepackage{enumerate}
\usepackage{mathrsfs}

\newbox\mybox
\def\overtag#1#2#3{\setbox\mybox\hbox{$#1$}\hbox to
  0pt{\vbox to 0pt{\vglue-#3\vglue-\ht\mybox\hbox to \wd\mybox
      {\hss$\ss#2$\hss}\vss}\hss}\box\mybox}
\def\undertag#1#2#3{\setbox\mybox\hbox{$#1$}\hbox to 0pt{\vbox to
    0pt{\vglue#3\vglue\ht\mybox\hbox to \wd\mybox
      {\hss$\ss#2$\hss}\vss}\hss}\box\mybox}
\def\lefttag#1#2#3{\hbox to 0pt{\vbox to 0pt{\vss\hbox to
      0pt{\hss$\ss#2$\hskip#3}\vss}}#1}
\def\righttag#1#2#3{\hbox to 0pt{\vbox to 0pt{\vss\hbox to
      0pt{\hskip#3$\ss#2$\hss}\vss}}#1}
\let\ss\scriptstyle

\def\Dot{\lower.2pc\hbox to 2pt{\hss$\bullet$\hss}}
\def\Circ{\lower.2pc\hbox to 2pt{\hss$\circ$\hss}}
\def\Vdots{\raise5pt\hbox{$\vdots$}}
\newcommand\lineto{\ar@{-}}
\newcommand\dashto{\ar@{--}}
\newcommand\dotto{\ar@{.}}

\allowdisplaybreaks[3]

\newtheorem{prop}{Proposition}[section]
\newtheorem{lem}[prop]{Lemma}

\newtheorem{cor}[prop]{Corollary}
\newtheorem{thm}[prop]{Theorem}
\newtheorem{rmk}[prop]{Remark}
\newtheorem{example}{Example}

\newtheorem{construction}[prop]{Construction}

\newtheorem{defn}[prop]{Definition}

            {\nolinebreak $\Box$ \end{trivlist}}

\newcommand{\noprint}[1]{}

\renewcommand{\tilde}{\widetilde}

\newcommand{\Ext}{\mbox{Ext}}

\newcommand{\XX}{{\mathfrak X}}

\renewcommand{\SS}{{\mathfrak S}}

\newcommand{\UU}{{\mathfrak U}}

\newcommand{\zz}{{\mathbb Z}}
\newcommand{\hh}{{\mathbb H}}

\newcommand{\T}{{\mathbb T}}

\newcommand{\qq}{{\mathbb Q}}
\newcommand{\pp}{{\mathbb P}}
\newcommand{\cc}{{\mathbb C}}
\newcommand{\rr}{{\mathbb R}}

\newcommand{\ff}{{\mathbb F}}

\newcommand{\sD}{{\mathcal D}}

\newcommand{\sE}{{\mathcal E}}

\newcommand{\sS}{{\mathcal S}}

\newcommand{\sO}{{\mathcal O}}

\newcommand{\sX}{{\mathcal X}}
\newcommand{\sY}{{\mathcal Y}}

\newcommand{\sV}{{\mathcal V}}
\newcommand{\sW}{{\mathcal W}}

\newcommand{\bbV}{{\overline{\overline{V}}}}
\newcommand{\bbW}{{\overline{\overline{W}}}}
\newcommand{\bV}{{\overline{V}}}
\newcommand{\bW}{{\overline{W}}}
\newcommand{\bbsV}{{\overline{\overline{\mathcal{V}}}}}

\DeclareMathOperator{\id}{id}

\DeclareMathOperator{\Aut}{Aut}

\DeclareMathOperator{\Pic}{Pic}

\DeclareMathOperator{\lci}{lci}

\DeclareMathOperator{\sing}{sing}
\DeclareMathOperator{\toric}{toric}
\DeclareMathOperator{\Star}{Star}

\newcommand{\rk}{\mathop{\rm rk}}

\newcommand{\tr}{\mathop{\rm tr}\nolimits}

\renewcommand{\top}{\mathop{\rm top}}

\renewcommand{\div}{\mathop{\rm div}}

\newcommand{\red}{\mathop{\rm red}\nolimits}

\newcommand{\proj}{\mathop{\rm Proj}\nolimits}
\newcommand{\tor}{\mathop{\rm tor}\nolimits}

\numberwithin{equation}{subsection}

\newcommand {\mat}      [1] {\left(\begin{array}{#1}}
\newcommand {\rix}          {\end{array}\right)}

\title[Equivariant smoothing of cusp singularities]{Equivariant smoothing of cusp singularities}
\author{Yunfeng Jiang}
\address{Department of Mathematics\\ University of Kansas\\ 405 Snow Hall 1460 Jayhawk Blvd\\Lawrence KS 66045 USA} 
\email{y.jiang@ku.edu}

\begin{document}
\sloppy \maketitle
\begin{abstract}
We generalize Looijenga's conjecture for smoothing surface cusp singularities to the equivariant setting. Moreover, we prove that for any cusp singularity which admits a one-parameter smoothing, the smoothing can always be induced by smoothing of locally complete intersection cusps.  The result provides evidence for the existence of the moduli stack of $\lci$ covers over semi-log-canonical surfaces. 
\end{abstract}

\maketitle

\tableofcontents

\section{Introduction}

A normal Gorenstein surface singularity $(\overline{V},p)$ is called a {\em cusp} if the exceptional divisor of the minimal resolution is a cycle of smooth rational curves or a rational nodal curve.    It is one type of minimally elliptic surface singularity in \cite{Laufer}.  Let $\pi: V\to \overline{V}$ be the minimal resolution, and let 
$$\pi^{-1}(p)=D=D_0+\cdots+D_{n-1}\in |-K_{V}|.$$
The analytic germ $(\overline{V},p)$ of the cusp singularity is uniquely determined by the self-intersections $D_i^2$ of $D$. 
When $n\ge 3$, we assume that $D_i\cdot D_{i\pm 1}=1$.  If $n=2$, $D$ is the union of two smooth rational curves that meet transversely at two distinct points. 
Since $D$ is contractible, Artin's criterion for contractibility implies that the intersection matrix $[D_i\cdot D_j]$ is negative-definite. 

Cusp singularities come naturally in dual pairs $(\overline{V},p)$ and $(\overline{V}^\prime,p^\prime)$.   Their resolution cycles $D$ and $D^\prime$ under minimal resolutions are called dual cusp cycles. 
From \cite{Inoue}, \cite{Hirzebruch},  there exists a compact, non-algebraic, analytic complex surface $V$ whose only curves are the dual cycles $D$ and $D^\prime$.   It is called the 
hyperbolic Inoue-Hirzebruch surface.  By contracting these two cycles we get a singular surface $\overline{\overline{V}}$ with only two cusp singularities $p, p^\prime$. 
We call $p$ the dual cusp of the cusp $p^\prime$, and vice versa.   We also call  $D$ the dual cycle of the cycle $D^\prime$, and vice versa. 
 
Recall from \cite{Looijenga}, an anti-canonical pair (called a Looijenga pair) $(Y,D)$ is a smooth rational surface $Y$, together with an anti-canonical divisor 
$D\in |-K_Y|$.  The topology and geometry of the Looijenga pairs were studied in \cite{Looijenga}, \cite{Friedman2}.  
The smoothing of cusp singularities  is close related to the geometry of Looijenga pairs. 
Looijenga studied the universal deformation of the hyperbolic Inoue-Hirzebruch surface $V$, and proposed the following conjecture, now a theorem:

\begin{thm}\label{thm_Looijenga_conjecture}(\cite[III Corollary 2.3]{Looijenga}, \cite{GHK15},\cite[\S 5]{Engel})
The cusp singularity $(\overline{V}, p^\prime)$ admits a smoothing if and only if the dual cycle $D$ of the dual cusp $p$ is the anti-canonical divisor of a Looijenga pair $(Y,D)$. 
\end{thm}

Looijenga \cite[III Corollary 2.3]{Looijenga} proved that there exists a universal deformation of the hyperbolic Inoue-Hirzebruch surface $V$ and gave a proof for the necessary condition of Theorem \ref{thm_Looijenga_conjecture}. 
The sufficient condition of Theorem \ref{thm_Looijenga_conjecture} was first proven by Gross-Hacking-Keel \cite{GHK15} using mirror symmetry.   For a Looijenga pair $(Y,D)$, Gross-Hacking-Keel constructed the mirror family of 
$(Y,D)$ using the log-Gromov-Witten invariants of $(Y,D)$, and the mirror family gives the smoothing of the dual cusp $p^\prime$ of $D$.   The natural existence of dual cusps $p, p^\prime$ on the Inoue-Hirzebruch surface implies the mirror symmetry property.   Later in \cite[\S 5]{Engel}, Engel gave a proof of the  sufficient condition of Theorem \ref{thm_Looijenga_conjecture}  using birational geometry--Type III degeneration of Looijenga pairs.  The proof is elegant and can be understood in a combinatorial way.   Both of these ideas were used to study the smoothing components of cusp singularities in \cite{FE}, and the compactification of KSBA moduli space of log Calabi-Yau surfaces and K3 surfaces, see \cite{AAB2024}, \cite{Alexeev_Engel}.

The main goal of this paper is to prove an equivariant version of Looijenga's conjecture.  Let us first discuss our motivation.  
For a cusp singularity germ $(\overline{W}, q)$, if the local embedded dimension is higher than $5$, then  \cite[Theorem 3.13]{Laufer} showed that the singularity is not a locally complete intersection (l.c.i.) singularity.  
In \cite[Proposition 4.1 (2)]{NW}, Neumann and Wahl  constructed a finite  cover 
$(\overline{V}, p)$ of $\overline{W}$ with finite transformation  group $G$ so that $(\overline{V}, p)$ is an l.c.i. cusp (actually a hypersurface cusp).  
The finite cover is determined by the link $\Sigma$ of the singularity germ $(\overline{W}, q)$, which is, by definition, the boundary of a neighborhood around the singularity $p$.   The link $\Sigma$ is a $T^2$-bundle over the circle $S^1$ and 
the first homology group 
$H_1(\Sigma,\zz)=\zz\oplus G^\prime$, 
where $G^\prime$ is the torsion subgroup of $H_1(\Sigma,\zz)$.  
There is a  surjective morphism $H_1(\Sigma,\zz)=\zz\oplus G^\prime\to G^\prime$   up to automorphisms of 
$\pi_1(\Sigma)$.  
The finite cover $\overline{V}$ in \cite[Proposition 4.1 (2)]{NW} was constructed using the monodromy matrix of the link and the 
discriminant cover determined by the morphism $H_1(\Sigma,\zz)=\zz\oplus G^\prime\to G^\prime$  above. We should point out that this finite cover may not be Galois in general.  It is interesting to see when a finite cover of a cusp is a Galois cover, see \cite{Anderson}.

Let $G$ be the finite transformation group of the finite cover. 
Then we have a quotient map:
\begin{equation}\label{eqn_quotient_of_cusp}
\mu:  (\overline{V}, p)\to (\overline{W}, q)
\end{equation}
such that 
$\overline{V}/G\cong \overline{W}$, and $\overline{V}\setminus \{p\}\to \overline{W}\setminus \{q\}$ is an unramified $G$-cover.

We are interested in the Gorenstein smoothings of  $(\overline{W}, p)$ which are induced from $G$-equivariant smoothings of the cusp  $(\overline{V}, q)$. 
For the cusp  $(\overline{W}, q)$, we denote its dual cusp by $q^\prime$, and the corresponding singular Inoue-Hirzebruch surface by 
$(\overline{\overline{W}}, q, q^\prime)$, and the compact complex analytic  Inoue-Hirzebruch surface by $(W, E, E^\prime)$. 

Since the two cusps $(\bbW, q, q^\prime)$ are dual to each other, we present the smoothing of the cusp $q^\prime$ in the following theorem. 
To state the theorem, we introduce a finite group action on Looijenga pairs. 
In the compactification of KSBA moduli space of log Calabi-Yau surfaces in  \cite{AAB2024}, the Looijenga pairs can be deformed to log Calabi-Yau surfaces with quotient singularities.  We define a finite group $G$-action on a Looijenga pair $(Y,D)$ (with negative definite self-intersection matrix $[D_i\cdot D_j]$) to be $hyperbolic$, if $G$ acts on $Y\setminus D$ only has quotient singularities, and acts on a neighborhood $V_D$ of $D$ as the action of neighborhood of the contracted cusp.

Our main result is:

\begin{thm}\label{thm_cusp_equivariant_smoothing}
The cusp $(\overline{V}, p^\prime)$ admits a $G$-equivariant smoothing such that the quotient space induces a smoothing of the cusp $(\overline{W}, q^\prime)$
if and only if the dual cycle $D$ of the cusp $p$ lies as an anti-canonical divisor in a Looijenga pair $(Y, D)$, and the pair $(Y, D)$ admits a  hyperbolic group $G$-action such that the $G$ action  on the complement 
$Y\setminus D$ only has quotient isolated singularities, and the quotient space $(Y, D)/G$, maybe after suitable resolution of singularities along $D/G$, gives a Looijenga pair $(X,E)$ such that $E$ is the dual cycle of the dual cusp $q$ of $q^\prime$. 
\end{thm}

To prove the necessary condition of the theorem, we use the fact that for any finite subgroup $G$ in the automorphism group of an Inoue-Hirzebruch surface $\bbV$, the quotient space is still an Inoue-Hirzebruch surface $\bbW$.
We  follow the method of \cite{Engel} to prove the sufficient condition, by putting the finite group action into the combinatorial construction of Type III canonical degeneration  pairs. 

Theorem \ref{thm_cusp_equivariant_smoothing} implies the following result.
\begin{thm}\label{thm_smoothing_lci_cusp_intro}(Theorem \ref{thm_smoothing_lci_cusp})
  Let $(\overline{V}^\prime,p^\prime)$ be a cusp singularity.   Suppose that $(\overline{V}^\prime,p^\prime)$ admits a smoothing $f: \overline{\sV}\to \Delta$.  Then there exists a smoothing $\tilde{f}: \sV\to \Delta$ of an lci cusp together endowed with  a finite group $G$ action such that the quotient induces the smoothing $f: \overline{\sV}\to \Delta$.
\end{thm}

Theorem \ref{thm_cusp_equivariant_smoothing} and  Theorem \ref{thm_smoothing_lci_cusp_intro} have applications for the moduli stack of $\lci$ covers defined  in \cite{Jiang_2022}, where the author constructed the virtual fundamental class on the 
moduli space of $\lci$ covers over the semi-log-canonical surfaces. 
Surface cusp singularities are semi-log-canonical singularities in the construction of the moduli space of  general type surfaces.   The KSBA compactification of the moduli space $\overline{M}_{K^2, \chi}$ of general type surfaces is the surface analogue of the moduli space of stable curves $\overline{M}_g$.  Here $K^2=K_S^2, \chi=\chi(\sO_S)$ for a surface in $\overline{M}_{K^2, \chi}$.   The boundary of $\overline{M}_{K^2, \chi}$ parametrizes singular surfaces with semi-log-canonical singularities; see \cite{Kollar-Shepherd-Barron} for more details.  Let us only talk about the normal surface case.   Then the log-canonical surface singularities, except the locally complete intersection singularities,   quotient singularities,  are given by 
simple elliptic singularities and cusp singularities.    The cusp singularities with higher embedded dimension $(>5)$ are bad surface singularities, which have higher obstruction spaces (see \cite{Jiang_2021}).
These singularities cause trouble in the construction of a perfect obstruction theory on the moduli space $\overline{M}_{K^2, \chi}$ in \cite{Jiang_2022}.  

In order to control such bad singularities, the author introduced the $\lci$ covering DM stack $\SS^{\lci}\to S$ for the semi-log-canonical surface $S$ with  cusp singularities of higher embedded dimension. 
The stacky structure around a cusp singularity $p\in S$ is given by the DM stack $[\widetilde{S}/G]$, where $(\widetilde{S}, p)\to (S,q)$ is the $\lci$ cover as in \cite[Proposition 4.1]{NW} and $G$ is the transformation group. 
Suppose that there is a $\qq$-Gorenstein flat family $\sS\to T$ of semi-log-canonical surfaces.  If there are cusp singularities with higher embedded dimension on the fiber surface $\sS_t$, then we can take the $\lci$ cover to get  the $\lci$ covering DM stack.  But it is hard to see if the $\lci$ covering DM stacks form a flat family.    If there is an  $\lci$ covering DM stack $\SS^{\lci}\to S$ to a semi-log-canonical surface $S$ with only cusp singularities,  then any 
$G$-equivariant smoothing of $\widetilde{S}$ locally gives a smoothing of the  $\lci$ covering DM stack $\SS^{\lci}\to S$.  Thus it gives a smoothing of the underlying semi-log-canonical surface $S$.  In 
\cite[\S 6]{Jiang_2022},  the author constructed the moduli stack $\overline{M}^{\lci}_{K^2, \chi}$ of $\lci$ covers over the KSBA component $\overline{M}_{K^2, \chi}$, by taking crepant resolution of smoothing of cusp singularities which becomes $\lci$ on the fiber surfaces.  Thus,  every $\qq$-Gorenstein family of semi-log-canonical surfaces in $\overline{M}_{K^2, \chi}$ can be obtained from a family of $\lci$ covering DM stacks.  The price we pay for taking crepant resolutions is that we only have a proper morphism $\overline{M}^{\lci}_{K^2, \chi}\to   \overline{M}_{K^2, \chi}$ from the moduli stack of lci covers to the  KSBA space.

Thus,  from Theorem \ref{thm_cusp_equivariant_smoothing} and  Theorem \ref{thm_smoothing_lci_cusp_intro} we have:

\begin{cor}  
The moduli space of $\lci$ covers over s.l.c. surfaces defined in \cite[\S 6]{Jiang_2022} is valid  if there are bad cusp singularities (with embedded dimension $>5$) on the semi-log-canonical surfaces $S$ in the moduli space 
$\overline{M}_{K^2, \chi}$. 
\end{cor}

Our Theorem \ref{thm_cusp_equivariant_smoothing} and  Theorem \ref{thm_smoothing_lci_cusp_intro} can only get one-parameter equivariant  smoothing of cusp singularities. Thus, we only get a curve inside the KSBA moduli space and the moduli stack of lci covers.  It is interesting to see the proper map property $\overline{M}^{\lci}_{K^2, \chi}\to   \overline{M}_{K^2, \chi}$ in the case of cusp singularities.   In a future work, we generalize the work \cite{AAB2024} in the equivariant setting and we hope to see the proper map on the moduli spaces.

The case of simple elliptic singularities is studied in \cite{Jiang_2023} using the geometry and topology  for the Minor fibre and link of the singularities in the smoothing. 

\subsection{Related work}

There are two classes of log canonical surface germ  singularities $(S,p)$.   The index of the singularity $p$ is, by definition, the least integer $N$ such that 
$\omega_S^{[N]}:=(\omega_S^{\otimes N})^{**}=\sO_S(N\cdot K_S)$ is invertible, where $K_S$ is the canonical class of $S$.  In the case of $N=1$,  the germ singularities $(S,p)$ are given by simple elliptic singularities and cusp singularities. 
Our study above focuses on the index one case, and a key point of Theorem \ref{thm_cusp_equivariant_smoothing} is that there exist  finite group 
$G$-actions on the Inoue-Hirzebruch surface 
$(\bbV, p,p^\prime)$ (only two points $p, p^\prime$ are fixed by $G$) such that the quotient $(\bbV, p,p^\prime)/G$ is still an  Inoue-Hirzebruch surface $(\bbW, q, q^\prime)$. 
In this setting we have to work on analytic complex surfaces. 

If the index $N>1$, then from \cite[Theorem 4.24]{Kollar-Shepherd-Barron}, the germ singularities $(S,p)$ are given by $\zz_2, \zz_3, \zz_4, \zz_6$ quotients of  simple elliptic singularities,  and  $\zz_2$ quotients of cusps, and  degenerate cusps. 
Degenerate cusp singularities are non-normal surface singularities. 
The $\zz_2$ quotient-cusp singularities are rational singularities, which can not be cusp singularities any more.   For such germ singularities, the link $\Sigma$ is a rational homology sphere.   Let $G=H_1(\Sigma,\zz)$ be the finite abelian group.   In \cite{NW}, Neumann-Wahl constructed a universal abelian cover $(\widetilde{S}, q)\to (S,p)$ which is a Galois $G$-cover. 
The germ $(\widetilde{S}, q)$ is a locally complete intersection cusp singularity.  The local defining equations of this  locally complete intersection cusp singularity  $(\widetilde{S}, q)$ were given in \cite[Theorem 5.1]{NW}.   Since a locally complete intersection  singularity admits a $G$-equivariant smoothing, its quotient gives the smoothing of the singularity $(S,p)$.  This also shows that our moduli stack of $\lci$ covers over s.l.c. surfaces defined in \cite[\S 5.3.6]{Jiang_2022} is valid if there are such quotient-cusp singularities. 

On the other hand, in \cite{Simonetti},  A. Simonetti studied the Looijenga conjecture for  $\zz_2$-equivariant smoothings of cusps singularities using log Gromov-Witten theory and  the techniques in \cite{GHK15}. 
Thus it is interesting to study the $G$-equivariant smoothings of $(\widetilde{S}, q)$ which induce smoothings of the  $\zz_2$ quotient-cusp $(S,p)$.   Since the quotient $(\widetilde{S},q)/G=(S,p)$ is not cusp singularity anymore, this provides difficulties for the construction of the corresponding Looijenga pairs. 

\subsection*{Acknowledgments}

Y. J. thanks Yuchen Liu and Ziquan Zhuang for  valuable discussion  on semi-log-canonical singularities.  Y. J. thanks P. Engel for his correspondence of the proof of Looijenga conjecture and the wonderful talk at Kansas, and Valery Alexeev, H\"ulya Arg\"uz, and Pierrick Bousseau for the valuable discussions. 
This work is partially supported by NSF DMS-2401484 and Simon Collaboration Grant.


\section{Inoue-Hirzebruch surfaces with a finite group action}\label{sec_Inoue-Hirzebruch}

We study the Inoue-Hirzebruch surfaces together with a finite group action. 

\subsection{Inoue-Hirzebruch surfaces}\label{subsec_Inoue-Hirzebruch}

Let us recall the construction of Inoue-Hirzebruch surface $V$ in \cite{Inoue}, \cite{Hirzebruch}, \cite[III, \S 2.]{Looijenga}, \cite{GHK15}. 

Let $M=\zz^2$ be a rank two lattice, and $\sigma\in SL_2(\zz)$ be a hyperbolic matrix.  Then $\sigma$ has two real eigenvalues $\lambda, \frac{1}{\lambda}$ for $\lambda>1$. 
Let $v_1, v_2$ be the two eigenvectors corresponding to $\lambda_1=\frac{1}{\lambda}, \lambda_2=\lambda$ so that $v_1\wedge v_2>0$.  
Let $\overline{C}, \overline{C}^\prime$ be two strictly convex cones spanned by 
$v_1, v_2$ and $v_2, -v_1$.   Let $C, C^\prime$ be their interiors which are both preserved by $\sigma$. 
Denote by 
$$\mathfrak{D}:=\{z=x+iy\in M_{\cc}/M| y\in \mathbb{H}\}$$
where $\mathbb{H}$ is the upper half-plane.  Then the finite cyclic group generated by $\sigma$ acts freely and properly discontinuously on $\mathfrak{D}$.  The quotient surface 
$V^{o}:=\mathfrak{D}/\langle \sigma\rangle$ is the open Inoue-Hirzebruch surface.  The compactification (Inoue-Hirzebruch surface) $\bbV:=V^{o}\cup\{p,p^\prime\}$ is obtained by adding two singular cusp points 
$p, p^\prime$.

Let 
$$U_{C}^\prime:=\{z=x+iy\in M_{\cc}/M| y\in C\}; \quad U_{C^\prime}^\prime:=\{z=x+iy\in M_{\cc}/M| y\in C^\prime\}$$
Then we have two neighborhoods $V_C^{o}:=U_{C}^\prime/\langle \sigma\rangle$ and $V_{C^\prime}^{o}:=U_{C^\prime}^\prime/\langle \sigma\rangle$ in $\bbV^\prime$.  
Let 
$$\bbV_C:=V_C^{o}\cup\{p\}; \quad  \bbV_{C^\prime}:=V_{C^\prime}^{o}\cup\{p^\prime\}.$$
Then $(\bbV_C, p)$ and $(\bbV_{C^\prime}, p^\prime)$ are the two singularity germs for the cusps $p$ and 
$p^\prime$, respectively in $\bbV$. 

Taking  the minimal resolutions of singularities for $p, p^\prime$, we get a smooth compact complex surface $V$ with two cycles of rational curves 
$$D=D_0+D_1+\cdots+D_{n-1}; \quad  D^\prime=D_0^\prime+ D_0^\prime+\cdots+D_{s-1}^\prime$$
corresponding to $p, p^\prime$ respectively.  
Let $V:=V^{o}\cup\{D, D^\prime\}$. 
Then we  call $V$ the Inoue-Hirzebruch surface with only cycles of curves $D$ and $D^\prime$. 
Let 
$$(d_0, d_1, \cdots, d_n); \quad   (d_0^\prime, d_1^\prime, \cdots, d_s^\prime)$$
be the cycle of $D$ and $D^\prime$ given by negative self-intersection numbers, where 
\begin{equation}\label{eqn_self_intersection}
d_i=
\begin{cases}
-D_i^2 & \text{~if~}n>0\\
2-D_i^2 & \text{~if~}n=0.
\end{cases}
\end{equation}
The numbers $d_i^\prime$ are defined similarly.   Since both $D$ and $D^\prime$ are contractible, the intersection matrix 
$[D_i\cdot D_j]$ is negative definite, which implies that 
$d_i\ge 2$ for all $i$, and $d_i\ge 3$ for some $i$. 
The generator $\sigma$ is given by:
$$\sigma=\mat{cc} 0&-1\\
1&d_{n-1}\rix \mat{cc} 0&-1\\
1&d_{n-2}\rix\cdots \mat{cc} 0&-1\\
1&d_0\rix$$

The duality property of the dual cusps $D$ and $D^\prime$ implies that the cycles $(d_0, \cdots, d_{n-1})$ and $(d^\prime_0, \cdots, d^\prime_{s-1})$ have the following properties:
if 
\begin{equation}\label{eqn_cycle_D}
(d_0, \cdots, d_{n-1})=(a_0+3, \underbrace{2, \cdots, 2}_{b_0}, \cdots, a_k+3, \underbrace{2, \cdots, 2}_{b_k})
\end{equation}
where $a_i, b_i\ge 0$. Then the negative self-intersections $(d^\prime_0, \cdots, d^\prime_{s-1})$ are obtained from $D$ by:
\begin{equation}\label{eqn_cycle_D_prime}
(d^\prime_0, \cdots, d^\prime_{s-1})=(b_0+3, \underbrace{2, \cdots, 2}_{a_0}, \cdots, b_k+3, \underbrace{2, \cdots, 2}_{a_k}).
\end{equation}

\subsection{Hirzebruch-Inoue modular surface}\label{subsec_Hirzebruch}

The Inoue surface can also be constructed as  Hilbert modular surfaces
 in \cite{Hirzebruch}.  
From \cite[\S 2]{Hirzebruch}, \cite[\S 2]{Pinkham}, let $K=\qq(\omega)$ be a real quadratic field and 
$$M=\zz\omega+\zz\cong \zz^2\subset K$$
be the lattice.  Then the hyperbolic matrix
$\sigma\in SL_2(\zz)$ determines the irrational number $\omega=\overline{[d_0, \cdots, d_{n-1}]}$. 
Let $\mathfrak{U}_M$ (rsp. $\UU_{M}^{+}$) be the group of positive (rsp. totally positive) units of 
$K$. Then $\sigma$ (which determines a cycle $(d_0, \cdots, d_{n-1})$) determines a totally positive unit in $\UU_{M}^{+}$.  We still denote this totally positive unit by $\sigma$, which 
generates a subgroup
$$\langle\sigma\rangle\subset \UU_M$$
with finite index. The $\sigma$ is the same as the element in $SL_2(\zz)$ in the beginning of Section \ref{subsec_Inoue-Hirzebruch}.

We explain how to add the cycles $D, D^\prime$ to $V^{o}$. From \cite[Page 302]{Pinkham}, for each $k\in \zz$, we take a $\cc^2$ given by coordinates 
$(u_k, v_k)$. We glue all the infinite $\cc^2$'s  (indexed by $k\in\zz$) by:
\begin{equation}\label{eqn_glue_coordinates}
\begin{cases}
u_{k+1}=u_k^{d_k}v_k;\\
v_{k+1}=\frac{1}{u_k}.
\end{cases}
\end{equation}
Let $A$ denote such a space.  Note that $D_k\cong \pp^1$ is given by $\{u_{k+1}=v_k=0\}$ and $D_k^{2}=-d_k$. 
The group $\langle\sigma\rangle$ acts on $A$ freely by:
$$\sigma(u_k, v_k)=(u_{k+n}, v_{k+n}).$$
Then we have an isomorphism:
$$\Phi: A-\bigcup_{k\in \zz}D_k\cong  U_{C}^\prime$$
given by:
\begin{equation}\label{eqn_glue_copies_C2}
\begin{cases}
2\pi i z_1=\omega \log u_0+\log v_0;\\
2\pi i z_2=\omega^\prime \log u_0+\log v_0
\end{cases}
\end{equation}
where $\omega=\overline{[d_0, \cdots, d_{n-1}]}$ is the irrational number which has a purely period modified fraction expansion, and 
$\omega^\prime$ is its conjugate.  (Here we also can identify $\zz^2$ (as a $\zz$-module) generated by $1,\omega$). 
Recall that $U_{C}^\prime:=\{z=x+iy\in \cc^2/M| y\in C\}$. 
We consider 
$\Phi^{-1}(\mathbb{H}\times \mathbb{H}/M)$.  Then  $\Phi$ is compatible with the action of $\langle\sigma\rangle$ on (\ref{eqn_glue_copies_C2}).   Gluing 
$A/\langle\sigma\rangle$ to $\mathbb{H}\times \mathbb{H}/M\rtimes \langle\sigma\rangle\subset U_C^\prime/\langle\sigma\rangle=V_C^{o}$
and we get a neighborhood of $D$ in $V$.  We denote this neighborhood by $V_C$.

We can do the similar to the cycle $D^\prime$. This is from the duality property of the cusp $p$ and its dual $p^\prime$.  Recall that $\omega=\overline{[d_0, \cdots, d_{n-1}]}$.
Let the modified continued fraction expansion of $1/\omega$ be
$$1/\omega=[f_1,\cdots, f_k,\overline{d^\prime_0, \cdots, d_{s-1}^\prime}].$$
We define the minimal period of $\omega$ as the least integer $k$ such that 
$d_{k+i}=d_{i}$ for any $i$ satisfying $1\le i\le n-k$.
Then from \cite[Deninition 1.6]{Nakamura}, 
$\omega^*=\overline{[d^\prime_0, \cdots, d^\prime_{s-1}]}$ is the dual irrational number of $\omega$, and $(d_0^\prime, \cdots, d_{s-1}^\prime)$ is the dual cycle of the dual cusp $p^\prime$. We also have 
$$\frac{n}{\text{minimal period of~}(d_0,\cdots, d_{n-1})}=
\frac{s}{\text{minimal period of~}(d^\prime_0,\cdots, d_{s-1}^\prime)}.$$

\begin{prop}(\cite[Lemma 4.4]{Nakamura})
Let $$\sigma=\mat{cc} 0&-1\\
1&d_{n-1}\rix \mat{cc} 0&-1\\
1&d_{n-2}\rix\cdots \mat{cc} 0&-1\\
1&d_0\rix$$
and 
$$\sigma^*=\mat{cc} 0&-1\\
1&d_{s-1}^\prime\rix \mat{cc} 0&-1\\
1&d^\prime_{s-2}\rix\cdots \mat{cc} 0&-1\\
1&d^\prime_0\rix$$
Then there exists an integral matrix $B$ such that $\det(B)=-1$ and 
$B\cdot \sigma=\sigma^*\cdot B$. In particular, 
$\det(\sigma-I)=\det(\sigma^*-I)$.
\end{prop}
Then we can use the same method to add the cycle $D^\prime$ and get the neighborhood  $V_{C^\prime}$.

\subsection{Finite group action on Inoue-Hirzebruch surfaces}\label{subsec_G_Inoue}

Let $\Aut(\bbV)$ be the automorphism group of the Inoue-Hirzebruch surface $\bbV$. 
Pinkham \cite{Pinkham}, and Prokhorov-Shramov \cite{PS} studied the automorphism group $\Aut(\bbV)$. 

We talk about  the $G$-action on the compact 
Inoue-Hirzebruch surface $V=V^{o}\cup\{D,D^\prime\}$ following \cite[\S 2]{Pinkham}.
Let 
$$\overline{M}:=(\sigma-1)^{-1}M$$
be a new $\zz$-module such that $\overline{M}\cong M$ since $(\sigma-1)\overline{M}=M$. Then from \cite[Theorem in \S 2]{Pinkham}, the full complex automorphism group 
of $V, \bbV$ is given by 
$$G(\overline{M}, \UU_M)/G(M, \langle\sigma\rangle)$$
such that our finite group $G\subset G(\overline{M}, \UU_M)/G(M, \langle\sigma\rangle)$. 
Here $G(\overline{M}, \UU_M)=\overline{M}\rtimes \UU_M$ and $G(M, \langle\sigma\rangle)=M\rtimes \langle\sigma\rangle$. Any finite group $G\subset G(\overline{M}, \UU_M)/G(M, \langle\sigma\rangle)$ must have the form 
$$G(N, W)/G(M, \langle\sigma\rangle)$$
for $N\subset \overline{M}$ and $W\subset \UU_M$.

\begin{prop}
    The part $\UU_M/\langle\sigma\rangle$ of the automorphism group naturally extends to  $V, \bbV$.  The quotient does not change the corresponding irrational number $\omega$. In particular, if $(d_0,\cdots, d_{n-1})$ is already a minimal period, then 
    the $\UU_M/\langle\sigma\rangle$ action on the Inoue surface is trivial.
\end{prop}
\begin{proof}
    This is from the fact that the fundamental group of of the Inoue surface $V$ is $\pi_1(V)=\zz$ in \cite[Proposition 4.1]{Inoue}.  Thus, for any $r\in \zz_{>0}$, there is a unique $r$-fold cover of $V$, which is still an Inoue surface but with the period $(d_0,\cdots, d_{n-1})$, and $(d_0^\prime,\cdots, d_{s-1}^\prime)$ rotating $r$-times  in cyclic orders. 
\end{proof}

For the part $\overline{M}/M$, for any  element $m\in \overline{M}/M$ representing an element in $G$, we have 
$m$ acts on $\hh\times \hh$ by 
$$m\cdot (z_1, z_2)=(z_1+m, z_2+m^\prime),$$
where $m^\prime$ is the conjugate of $m$.   By 
  equation (\ref{eqn_glue_copies_C2}), 
$m$ acts on the $0$-th coordinate chart of $A$ by 
\begin{equation}\label{eqn_action1}
m\cdot (u_0, v_0)=(e^{2\pi i m_1}u_0, e^{2\pi i m_2}v_0)
\end{equation}
where $m=m_1\omega+m_2$, $m_1, m_2\in \qq$,   and 
$\omega=\overline{[d_0, \cdots, d_n]}$, $M\cong M(\omega)$ which is generated by $1, \omega$.
By iterating equation  (\ref{eqn_glue_coordinates}), we get 
$$
\begin{cases}
u_{r}=u_0^{p_r}v_0^{q_r};\\
v_{r}=u_0^{-p_{r-1}}v_0^{-q_{r-1}}
\end{cases}
$$
where 
$$\mat{cc} p_r&q_r\\
-p_{r-1}&-q_{r-1}\rix =\mat{cc} d_{r-1}&1\\
-1&0\rix \mat{cc} d_{r-2}&1\\
-1&0\rix\cdots \mat{cc} d_0&1\\
-1&0\rix.$$
Set 
$$N:=\mat{cc} p_r&q_r\\
-p_{r-1}&-q_{r-1}\rix.$$
By \cite[Step II]{Pinkham}, $(\omega, 1)N=\sigma(\omega, 1)$.  So $m$ acts on the $r$-th coordinate chart of $A$ by
\begin{equation}\label{eqn_action2}
m\cdot (u_r, v_r)=(e^{2\pi i (p_r m_1+q_r m_2)}u_r, e^{2\pi i (-p_{r-1}m_1-q_{r-1} m_2)}v_r).
\end{equation}
Then the $\overline{M}$-action descends to the quotient $\overline{M}/M$ on $A/\langle\sigma\rangle$.  Thus, 
$G\subset G(\overline{M}, \UU_M)/G(M, \langle\sigma\rangle)$ acts on $V_C$. 
Although the quotient $V_C/G$ may not be the resolution of a cusp singularity, 
a resolution of the quotient singularities will give a cusp.

\begin{example}\label{example_cusp_1}
Consider the hypersurface cusp $\{x^3+y^3+z^5+xyz=0\}$ whose negative self-intersection sequence of the minimal resolution cycle is given by
$(5,2)$. One can directly calculate this cycle by taking two time blow-ups along the origin of the hypersurface. After the first blow-up, we get a rational $A_1$-nodal curve as the exceptional divisor; then blowing up along the $A_1$-singularity, we get the minimal resolution of the cusp.    
In this case, from \cite[Example]{Pinkham},  one possible choice for $\omega = \overline{[2,5]}= 1 + \sqrt{15}/5$,$\sigma = 4 + \sqrt{15}$ and $\sigma-1$ has norm $-6$.  

The cycle $(5,2)$ is a minimal period, and hence $\UU_M/\sigma$ acts trivially. 
$\overline{M}/M\cong \mu_6$. 
The cyclic group $\mu_6=\langle\zeta\rangle$ acts on the cusp by:
$$x\mapsto \zeta x; \quad y\mapsto \zeta^5 y; \quad z\mapsto \zeta^3 z.$$
Then from \cite[Example on Page 308]{Pinkham} the quotient is given by the hypersurface cusp
$\{x^2+y^4+z^7+xyz=0\}$. One can calculate  the negative self-intersection sequence of the minimal resolution cycle of this cusp which  is given by  $(4,2,2)$, see Figure \ref{fig-cusp1-config}. 

\begin{figure}[htbp]
  \centering
  \includegraphics[width=50mm]{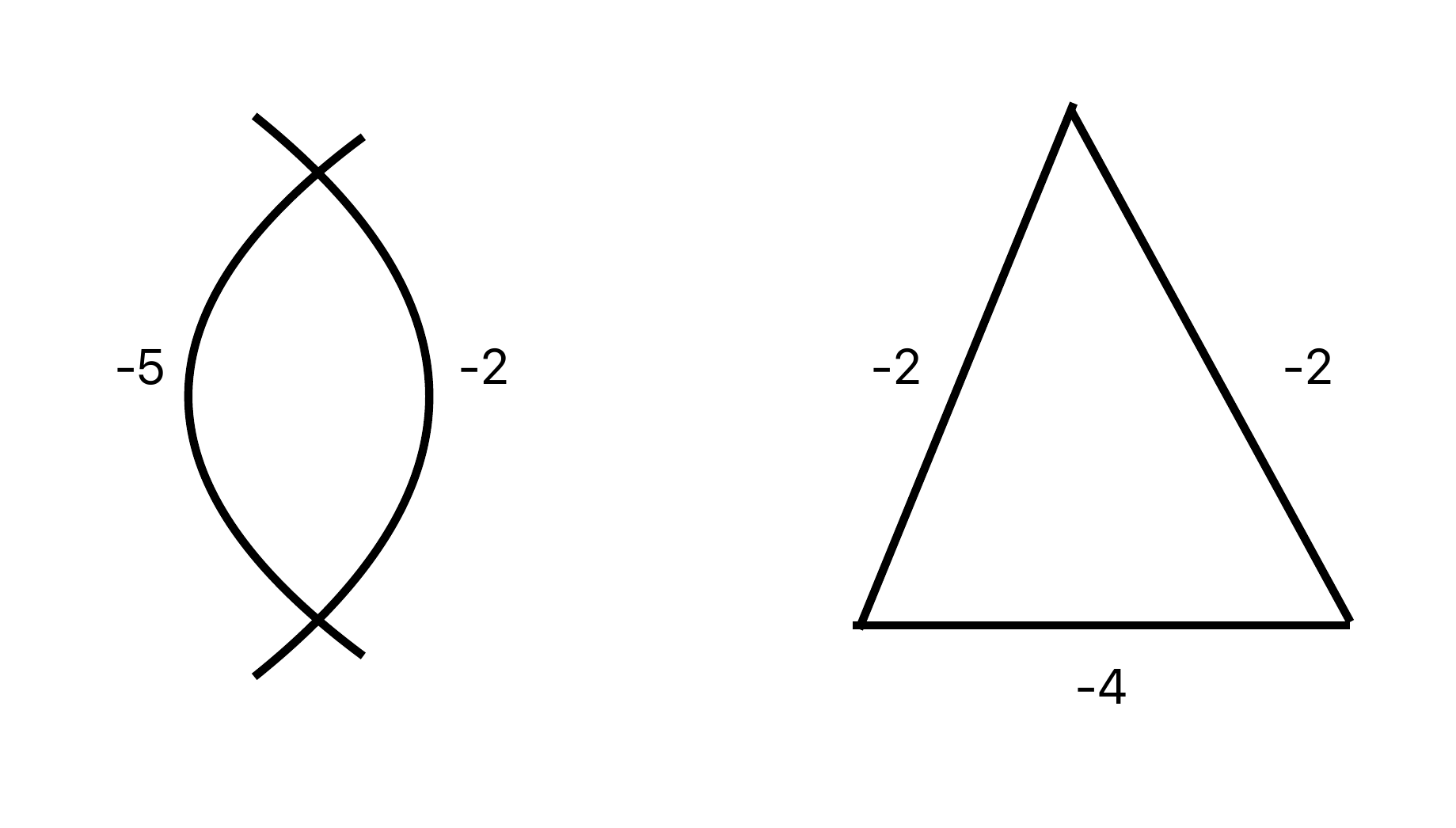}      
  \caption{cusp $(5,2)$ and its dual}
  \label{fig-cusp1-config}
\end{figure}

To get the minimal resolution $(4,2,2)$,  under the $\mu_6$-action on the exceptional divisor $(5,2)$ with two fixed points the corners, we take resolution and get the resolution in    Figure \ref{fig-cusp3-config}.  Then blowing down the $(-1)$-curve we get the minimal resolution $(4,2,2)$.

\begin{figure}[htbp]
  \centering
  \includegraphics[width=40mm]{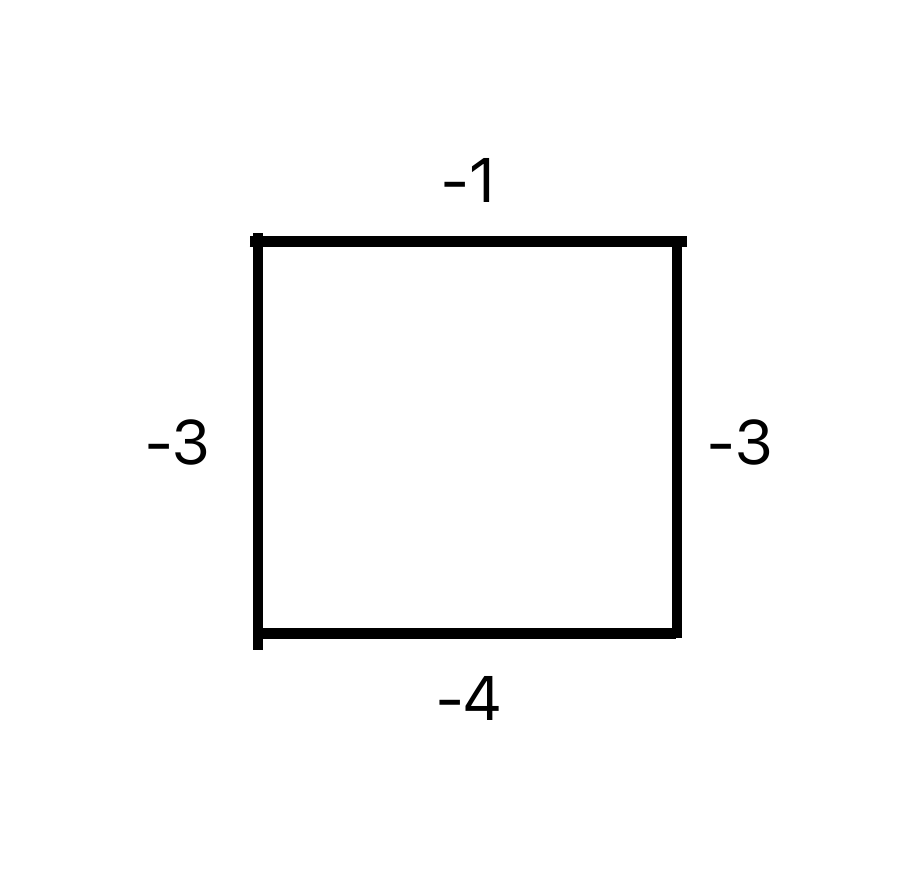}      
  \caption{resolution $(-4,-3,-3,-1)$}
  \label{fig-cusp3-config}
\end{figure}

Let  $\mu_3=\langle\zeta^2\rangle$, and it acts on the cusp $\{x^3+y^3+z^5+xyz=0\}$ in an obvious way.  The quotient gives the cusp 
whose  minimal resolution cycle has negative self-intersection sequence  $(3, 2,2,2,2,2)$, see Figure \ref{fig-cusp2-config}.

\begin{figure}[htbp]
  \centering
  \includegraphics[width=40mm]{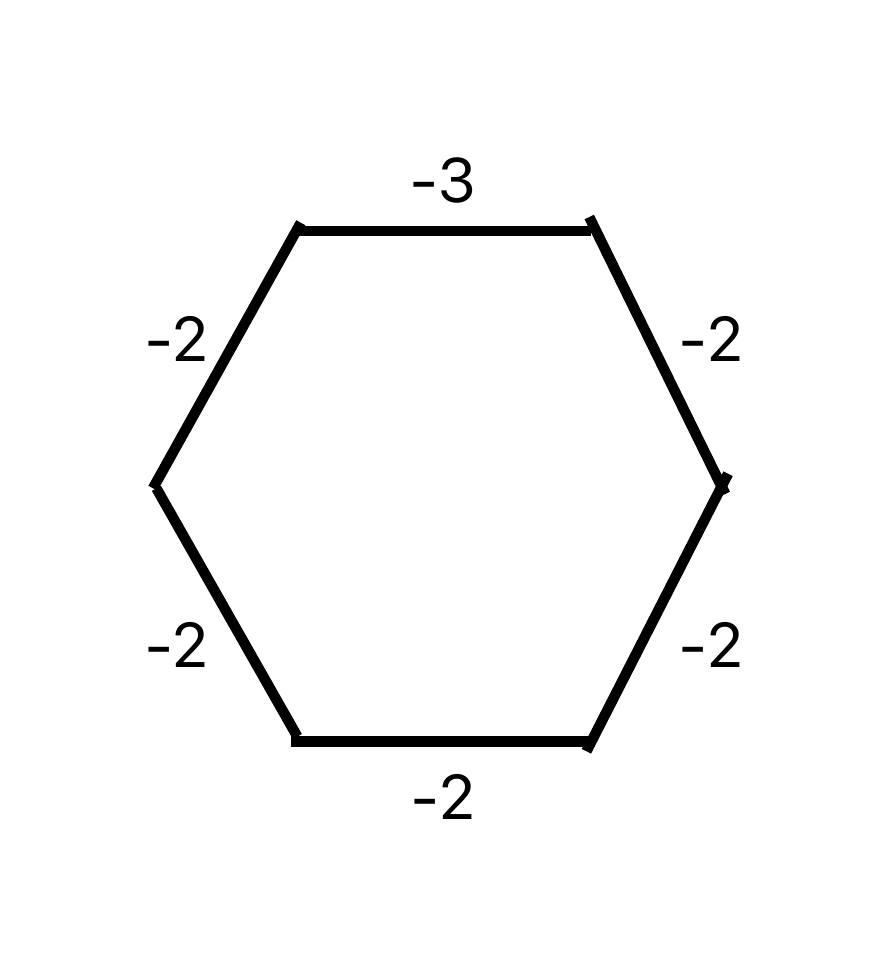}      
  \caption{cusp $(3,2,2,2,2,2)$}
  \label{fig-cusp2-config}
\end{figure}

Similarly, let $\mu_2=\langle\zeta^3\rangle$ be the order $2$ cyclic group.  Its action on the above cusp has the quotient which is the cusp whose minimal resolution cycle has 
negative self-intersection sequence $(2+6)$. This is a complete intersection cusp whose local equation is given by 
$$x^2+w^2=yz, \quad  y^2+z^3=xw$$
from \cite[lemma 2.5]{Nakamura}.

So the self-intersection of the rational nodal curve is $-6$. One can see this from the action of $\mu_2$ on the divisor exceptional divisor $(5,2)$ in (\ref{eqn_action1}) and (\ref{eqn_action2}). The quotient of $V_C$ (which is a neighborhood of the exceptional divisor $(5,2)$) by $\mu_2$ makes the divisors into $(10, 1)$, see Figure \ref{fig-cusp4-config}.  This contains two rational $\pp^1$'s intersecting in two points, and one with self-intersection number $-10$, and the other with  self-intersection number $-1$.  Then we contract the $-1$-curve and get a rational nodal curve with self-intersection number $-6$.   

Here is another way to see the resolution cycle.  Let $(V,0)$ be the germ of the cusp $\{x^3+y^3+z^5+xyz=0\}$. $(V,0)/\mu_2$ is the complete intersection cusp above.  Then let $Bl_{0}(V)$ be the blow up along the origin, then it has an exceptional curve which is a rational nodal curve with $A_1$-singularity at the node and self-intersection number $-3$.  Since $Bl_{0}(V/\mu_2)\cong Bl_{0}(V)/\mu_2$, then $Bl_{0}(V)/\mu_2$ gives the rational nodal curve with self-intersection $-6$. 

\begin{figure}[htbp]
  \centering
  \includegraphics[width=50mm]{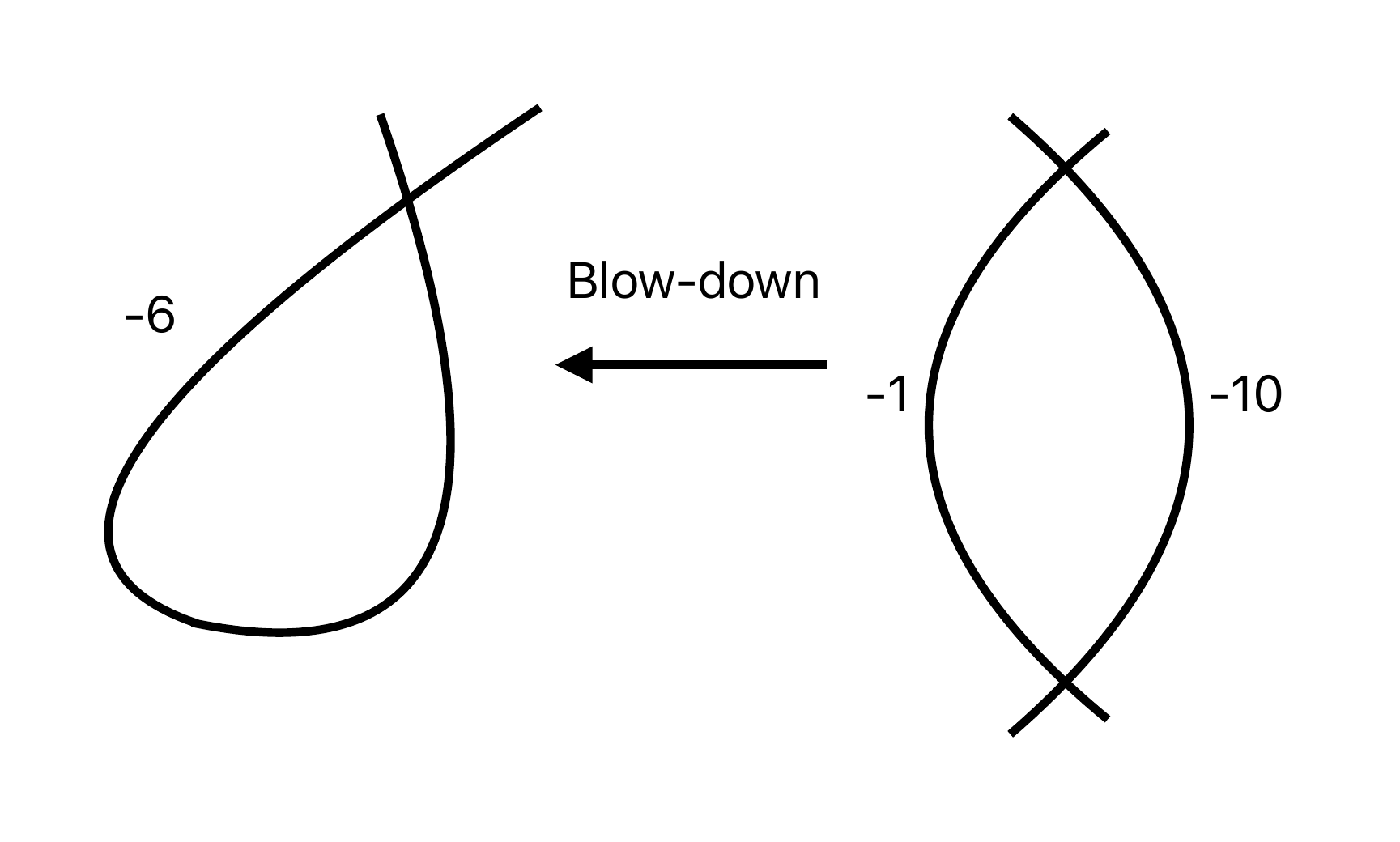}      
  \caption{blow-down process}
  \label{fig-cusp4-config}
\end{figure}
\end{example}

\section{Equivariant anticanonical pairs}\label{sec_Looijenga_pair}

\subsection{Looijenga pairs}\label{subsec_Looijenga_Pair}

We introduce basic knowledge of  Looijenga pairs.  More details can be found in \cite{Friedman}. 

\begin{defn}\label{defn_Looijenga_pair}
A Looijenga pair $(Y,D)$ is given by a smooth rational surface $Y$, together with a connected singular nodal curve $D\in |-K_Y|$. 
\end{defn}

Since the arithmetic genus $p_a(D)=1$, $D$ is either an irreducible rational nodal curve, or a cycle of smooth rational curves.  For such a $D$, we have 
$H_1(D,\zz)=\zz$.  Thus fixing a generator of $H_1(D,\zz)$ gives an orientation on $D$, and we label $D$ as 
$D=D_0+\cdots+D_{n-1}$.  The length of $D$ is $\ell(D)=n$.  
We call $D$ negative-definite, if the intersection matrix $[D_i\cdot D_j]$ is negative-definite. 

Similar to (\ref{eqn_self_intersection}), we define 
$$
d_i=
\begin{cases}
-D_i^2 & \text{~if~}n>0\\
2-D_i^2 & \text{~if~}n=0.
\end{cases}
$$
Then the negative self-intersection sequence is $\mathbf{d}=(d_0, \cdots, d_{n-1})$.

\begin{defn}\label{defn_charge_pair}
The charge of a Looijenga pair $Q(Y,D)$ is defined by the formula:
$$Q(Y,D):=12+\sum_{i=1}^{n-1}(d_i-3)=12+\sum_{i=1}^{n-1}(a_i-b_i).$$
\end{defn} 

Let $D^\prime$ be the dual cycle of $D$, then $Q(Y,D^\prime)$ is given by interchanging the $a_i$ with $b_i$ from (\ref{eqn_cycle_D}) 
and (\ref{eqn_cycle_D_prime}). We have 
$Q(Y,D)+Q(Y,D^\prime)=24$.  
A Looijenga pair $(Y, D)$ is called a toric pair if $Y$ is a toric variety, and $D$ is its boundary divisor.

There are good properties of the charge $Q(Y,D)$ for a Looijenga pair $(Y,D)$, which is related to the geometry and topology property of the pair.  We only mention some useful ones, and more properties can be found in
 \cite{Friedman2}, \cite[\S 4]{FM}.
If $D$ is negative-definite,  then $Q(Y, D)\ge 3$.  
Also 
\begin{prop}\label{prop_pair_Euler}(\cite[Lemma 1.2]{Friedman2})
For a Looijenga pair $(Y,D)$,  the charge 
$$Q(Y,D)=\chi_{\top}(Y-D)$$
\end{prop}

The charge of the Looijenga pair measures how far the pair is being a toric pair.  The above result implies that 
\begin{prop}\label{prop_toric_model}(\cite[Lemma 2.7]{Friedman2})
If $(Y,D)$ is an anticanonical pair, then $Q(Y,D) \ge 0$ and $(Y,D)$ is toric if and only if $Q(Y,D) =0$.
\end{prop}

From \cite[Definition (4.4)]{FM}, suppose that $D$ represents a cusp, meaning that it is the resolution cycle of a cusp singularity,  then we say that $D$ is rational if there exists a Looijenga pair $(Y, D)$ so that $D$ is the  anti-canonical divisor of  the rational surface $Y$.  If $D$ is a cusp such that the dual $D^\prime$ to $D$ is rational, we say that $D$ has a rational dual. 

\begin{prop}\label{prop_rational_dual} (\cite[Theorem (4.5)]{FM})
If $D$ has a rational dual, then the charge $Q(D)\leq 21$. 

Conversely,  if $D$ is a cusp and $\ell(D)\le 3$, $Q(D)\le 21$, then $D$ has a rational dual except in the following cases: $(4,11), (7,8), (2,4,12), (2,8,8)$, 
$(3,3,12), (3,4,11), (3,7,8),
(4,4,10), (4,6,8), (4,7,7), (5,5,8)$. 
\end{prop}

\subsection{Corner blow-ups and internal blow-ups}\label{subsec_corner_internal}

We introduce corner blow-ups and internal blow-ups for the Looijenga pairs. 
For a Looijenga pair $(Y,D)$, we can contract  an exceptional curve $E$ and get  a new Looijenga pair
$$\pi: (Y, D)\to (\overline{Y}, \overline{D}).$$
If $E\subset D$ is a component, then $E$ contracts to a node of the cycle $\overline{D}$. We call this type blow-up the corner blow-up. 
If $E\not\subset D$, then $E$ meets with $D$ at a smooth point of $D$. Then in this case $E$ contracts to a smooth point of the cycle $\overline{D}$. We call this type of blow-up the internal blow-up. 

\begin{prop}\label{prop_internal_corner_blow-up_charge}(\cite[Lemma 2.2]{Friedman2})
If there is a Looijenga pair $(Y, D)$ such that the negative self-intersection sequence is 
$(d_0, \cdots, d_{n-1})$ and the charge is $Q(Y, D)$, then we have 
\begin{enumerate}
\item  Let $\widetilde{Y}\to Y$ be an internal blow-up at the point 
$p\in D_i^{\circ}$ (the interior part of $D_i$),  then $\ell(\widetilde{D})=\ell(D)$, and, under the natural labeling of $\widetilde{D}$, the negative self-intersection sequence of 
$(\widetilde{Y}, \widetilde{D})$ is $(d_0, \cdots, d_{i-1}, d_i+1, d_{i+1},\cdots, d_{n-1})$.
\item If  $\widetilde{Y}\to Y$ is a corner blow-up of $Y$ at the point $p\in D_i\cap D_{i+1}$, then 
$\ell(\widetilde{D})=\ell(D)+1$.  If $\ell(D)=1$, i.e., $D$ is irreducible, then the negative self-intersection sequence of $(\widetilde{Y}, \widetilde{D})$ is
$(d_1+4, 1)$.  If $\ell(D)\ge 2$, and for an appropriate labeling of the components of $\widetilde{D}$,   the negative self-intersection sequence of $(\widetilde{Y}, \widetilde{D})$ is
$$(d_0, \cdots, d_{i}+1, 1, d_{i+1}+1, \cdots, d_{n-1}).$$ 
\item  If $\widetilde{Y}\to Y$ is an internal blow-up of $Y$,  then $Q(\widetilde{Y}, \widetilde{D})=Q(Y, D)+1$; and, 
if $\widetilde{Y}\to Y$ is a corner blow-up of $Y$,  then $Q(\widetilde{Y}, \widetilde{D})=Q(Y, D)$. 
\end{enumerate}
\end{prop}

From \cite[Proposition 1.3]{GHK15}, if we have a Looijenga pair $(Y,D)$, there exists a sequence of corner blow-ups $(Y^\prime, D^\prime)$ such that  $(Y^\prime, D^\prime)$ has a toric model.
This means  $(Y^\prime, D^\prime)$ can be obtained from a toric Looijenga pair $(\overline{Y}, \overline{D})$ by internal blow-ups at some number of smooth points. 

\subsection{The minimal model of Looijenga pairs with a finite group action}
We introduce finite group action on Looijenga pairs.  
Let us first recall the notion of $G$-minimal pairs. Our references for the group action on projective rational surfaces are \cite{Isk79}, \cite{Zhang}.

\begin{defn}\label{defn_G-minimal_pair}
   A $G$-equivariant birational morphism $\sigma: Y_1\to Y_2$ is a birational morphism satisfying $\sigma(g x)=g\sigma(x)$ for $g\in G$. 
   The existence of such a morphism $\sigma$ is equivalent to that of a $G$-stable divisor on $Y_1$ which can be smoothly blown down.   
   Therefore, this means that if $\Pic(Y)^G\otimes \qq$ has rank $1$, then there is no such a $\sigma$.

   Two pairs of finite group actions $(Y_i,G)$ are birationally $G$-equivariant if there is a birational map 
   $Y_1\to \cdots\to Y_2$ which can be decomposed as $f_1\circ \cdots\circ f_n$ such that for each $i$, either $f_i$ or $f_i^{-1}$ is a $G$-equivariant birational morphism. The action $(Y,G)$ is called a $G$-minimal pair if for any $G$-equivariant birational morphism 
   $\sigma: Y\to Y_1$, $\sigma$ must be identity $\id$.
\end{defn}

We have the following result for the $G$-minimal pairs as in \cite{Zhang}.

\begin{thm}\label{thm_Zhang_G}(\cite[Theorem 4]{Zhang})
    Suppose that $Y$ is a smooth $G$-minimal projective rational surface.  Then we have the following:
    \begin{enumerate}
        \item Suppose that the $G$-invariant sublattice $(\Pic(Y))^G$ has rank $\ge 2$, then $Y$ has a $G$-stable conic fibration over $\pp^1$ and each singular fiber is a linear chain of two $(-1)$-curves. 
        \item Suppose that the $G$-invariant sublattice $(\Pic(Y))^G$ has rank $1$, then $Y$ is a smooth del Pezzo surface and the quotient 
        $X=Y/G$ is a singular del Pezzo surface with at worst quotient singularities so that 
        $\pi_1(X^{0})$ is finite; one has $\pi_1(X^{0})=G$ if the fixed locus $Y^G$ is a finite set.  Moreover, 
        \begin{enumerate}
            \item If the Fano index $r(X)=1$, and $Y^G$ is a finite set, then modulo $G$-equivariant isomorphism, $Y$ is one of the surfaces in \cite[Examples 2.1b, 2.5, 2.9-2.11]{Zhang}.  All of the singularities in $X$ are in \cite[Table 2]{Zhang}.
            \item  If the Fano index $r(X)>1$, then $X$ is either $\pp^2$ or the projective cone $\overline{F}_e (e\ge 2)$ with $e| |G|$.
        \end{enumerate}
    \end{enumerate}
\end{thm}

From \cite[Theorem 2.4]{Friedman}, we have

\begin{prop}\label{prop_Friedman_pair}
If $(Y,D)$ be a minimal anti-canonical pair. Then exactly $(Y,D)$ is one of the following:
 \begin{enumerate}
    \item  $Y=\pp^2$, and $D$ is either three lines in general position, a line and a conic meeting transversally, or an irreducible nodal cubic. Equivalently, the possible negative self-intersection sequences are $(-1,-1,-1), (-1,-4)$, or $(-9)$. The corresponding values of the charges $Q(\pp^2,D)$ are: $Q(\pp^2,D) = 0$ if $D$ has negative  self-intersection sequence $(-1,-1,-1)$; $Q(\pp^2,D) = 1$ if $D$ has negative  self- intersection sequence $(-1,-4)$; $Q(\pp^2,D) = 2$ if $D$ has negative self-intersection sequence $(-9)$.
    \item  $Y=\ff_a$, $a\neq 1$, and $D$ is the union of the negative section $s_0$ and
    \begin{enumerate}
        \item  A section $s$ with $s^2 = a$ and two fibers $f_1$ and $f_2$. In this case,
the negative  self-intersection sequence is $(a, 0, -a, 0)$ and $Q(\ff_a , D) = 0$.
\item  A section $s$ with $s^2 = a + 2$ meeting $s_0$ transversally and one fiber $f$ not passing through the intersection points of $s_0$ and $s$. In this case, the negative  self-intersection sequence is $(a, -a-2, 0)$ up to orientation and $Q(\ff_a , D) = 1$.
\item  A section $s$ with $s^2 = a + 4$ meeting $s_0$ transversally. In this case, the negative self-intersection sequence is $(a,-a-4)$ and $Q(\ff_a,D) = 2$.
\end{enumerate}
\item  $Y= \ff_a$, $a = 0,2$, and $D$ is either an irreducible nodal bisection of negative self-intersection $-8$, with $Q(\ff_a,D) = 3$, or $D$ is the union of two sections of negative self-intersection $-2$, with $Q(\ff_a , D) = 2$.
\end{enumerate}
\end{prop}

Combining the above results and Theorem \ref{thm_Zhang_G}, the $G$-minimal pairs $(Y, D, G)$ are classified. 
\begin{prop}\label{prop_G-minimal_pair_cusp}
    Let $(Y,D)$ be a smooth Looijenga pair.  Suppose that there is a $G$-action on $(Y,D)$ such that $(Y,D)$ is $G$-minimal, then we have the cases: 
\begin{enumerate}
        \item Suppose that the $G$-invariant sublattice $(\Pic(Y))^G$ has rank $\ge 2$, then $Y$ has a $G$-stable conic fibration over $\pp^1$ and each singular fiber is a linear chain of two $(-1)$-curves. 
        The number of the $\pp^1$-components in the boundary divisor $D\in |-K_Y|$ can be $4, 5$ or $6$.  
        \item Suppose that the $G$-invariant sublattice $(\Pic(Y))^G$ has rank $1$, then $(Y,D)$ is a smooth del Pezzo surface pair and the quotient 
        $X=Y/G$ is a singular del Pezzo surface with at worst quotient singularities so that 
        $\pi_1(X^{0})$ is finite; one has $\pi_1(X^{0})=G$ if the fixed locus $Y^G$ is a finite set. The boundary divisor $E\in |-K_X|$
        is induced from $D$. Moreover, 
        \begin{enumerate}
            \item If the Fano index $r(X)=1$, and $Y^G$ is a finite set, then modulo $G$-equivariant isomorphism, $Y$ is one of the surfaces in \cite[Examples 2.1b (only $\mu_3$), 2.5, 2.9-2.11]{Zhang}.  All of the singularities in $X$ are in \cite[Table 2]{Zhang}.
            \item  If the Fano index $r(X)>1$, then $X$ is either $\pp^2$ or the projective cone $\overline{F}_e (e\ge 2)$ with $e| |G|$.
            The boundary divisor $E\in |-K_X|$ is  induced from $D$.  
        \end{enumerate}
    \end{enumerate}  
\end{prop}
\begin{proof}
From Theorem \ref{thm_Zhang_G}, if the  $G$-invariant lattice $(\Pic(Y))^G\otimes \qq$  of the $G$-minimal pair of $(Y,D)$
     has rank $\ge 2$, then $Y$ must be in the case (1) of Theorem \ref{thm_Zhang_G}.  The boundary divisor $D\in |-K_Y|$ must the zero section $\pp^1$ and infinity $\pp^1$, plus the two fibers.  Since the singular fiber only contains two $-1$-curves,  the possible components in $D$ only has $4, 5$ or $6$ components.

    If the  $G$-invariant lattice $(\Pic(Y))^G\otimes \qq$  of the $G$-minimal pair of $(Y,D)$
     has rank $1$, then $Y$ must be smooth del Pezzo surface.  
     This is from \cite{Manin}. 
     Therefore, it must belong to case (2) in Theorem \ref{thm_Zhang_G}.

     In the case that the quotient $r(X)=1$, 
     we list the examples Examples 2.1b, 2.5, 2.9-2.11] in \cite{Zhang}, and the boundary divisors. 
     \begin{enumerate}
         \item \cite[2.1 b]{Zhang}, in this case $Y=\pp^2$, $G=\mu_3$ acts on $\pp^2=\proj(\cc[x:y:z])$ by 
         $$
        [x:y:z]\mapsto [x, \zeta y, \zeta^2 z]
         $$
         where $\zeta\in \mu_3$ is the generator.
         $X=Y/\mu_3$ has three isolated fixed points $[1:0:0], [0:1:0], [0:0:1]$ and $X$ is a log del Pezzo surface of rank one. 
         The boundary divisor $D\in |-K_Y|$ is in (1) of Proposition \ref{prop_Friedman_pair} and $E\in |-K_X|$ is the quotient of $D$. 
         \item \cite[2.5]{Zhang}, in this case $Y=Bl_{4}\pp^2$ is the degree $5$ del Pezzo surface. There are $10$ $(-1)$-curves in $Y$. $D\in |-K_Y|$ is the $5$ of them intersecting transversally. 
The group $\mu_5=\langle g\rangle$ acts on $Y$ and $Y^{\mu_5}$ has two fixed points with type $\frac{1}{5}(1,4)$. The quotient $X=Y/\mu_5$ is a Gorenstein log del Pezzo surface with two type  $\frac{1}{5}(1,4)$ singularities.
The quotient $E=D/\mu_5$ is a rational quotient nodal curve with self-intersection number $-1$. The quotient nodal point has type $\frac{1}{5}(1,4)$.  We have $K_X^2=1$ and $E^2=1$.
\item \cite[2.9]{Zhang}, in this case $Y=Bl_{3}\pp^2$ is the degree $6$ del Pezzo surface.  $D\in |-K_Y|$ contains $6$ $-1$-curves intersecting transversally. 
The group $\mu_6=\langle g\rangle$ acts on $Y$ and $Y^{\mu_6}$ has three fixed points with type $\frac{1}{2}(1,1), \frac{1}{3}(1,2)$, and $\frac{1}{6}(1,5)$ respectively. The quotient $X=Y/\mu_6$ is a Gorenstein log del Pezzo surface with three types $\frac{1}{2}(1,1), \frac{1}{3}(1,2)$, and $\frac{1}{6}(1,5)$ singularities.
The quotient $E=D/\mu_6$ is a rational quotient nodal curve with self-intersection number $-1$.  We have $K_X^2=1$ and $E^2=1$.
\item \cite[2.10]{Zhang}, in this case $Y=\pp^2$, and 
$G=\zz_3\oplus \zz_3$ acts on $\pp^2$ effectively.  $Y^G$ is a finite set of order $12$. $X=Y/G$ is a log surface with $12$ $A_2$-singularities.
\item \cite[2.11]{Zhang}, in this case $Y=\pp^1\times\pp^1$, and 
$G=\zz_4$ acts on $\pp^1\times\pp^1$ effectively.  $Y^G$ has three singularities of type $\frac{1}{2}(1,1),\frac{1}{4}(1,3),\frac{1}{4}(1,3)$. 

In the second case $Y=\pp^1\times\pp^1$, and 
$G=\zz_2\oplus \zz_4$ acts on $\pp^1\times\pp^1$ effectively.  $Y^G$ is a finite set of order $4$. $X=Y/G$ is a log surface with $2$ $A_1$-singularities and $2$ $A_3$-singularities.
     \end{enumerate}
\end{proof}

\subsection{Finite group action on Looijenga pairs}\label{subsec_Looijenga_Pair_G}

In this section we consider  Looijenga pairs $(Y,D)$ with the intersection matrix $(D_i\cdot D_j)$ negative-definite.  This means that $D$ can be contracted to a cusp singularity.  We take the Looijenga pair $(Y,D)$ as analytic surfaces with analytic divisors. 
It is also interesting to study the finite group action on algebraic Looijenga pairs, and in general the action will have fixed points as quotient singularities. 

In general  Looijenga pair $(Y,D)$ may deform to Looijenga pair containing surface singularities.  There are many ways to think about the deformation of Looijenga pairs in the moduli space, including $K$-stable moduli space as in \cite{ADL2023}, KSBA stable pair moduli spaces in \cite{AAB2024}, and purely moduli of Calabi-Yau spaces as in  \cite{BL2024}. In either of these type moduli spaces, the singularities in the boundary of the deformation of Looijenga pairs may be different.  Here we consider general Looijenga pairs with at most quotient singularities.

\begin{defn}\label{defn_Looijenga_pair_G}
Let $(Y,D)$ and $(Y^\prime, D^\prime)$ be two Looijenga pairs.  An isomorphism between these two Looijenga pairs is given by  an isomorphism 
$$f: Y\to Y^\prime$$
such that $f(D_i)=D^\prime_i$ for $i=1, 2, \cdots, n$, and 
$f$ is compatible with the orientation of $D$ and $D^\prime$.   We let 
$\Aut(Y,D)$ be the automorphism group of $(Y,D)$. 

Let $G$ be a finite group.  We say a Looijenga pair $(Y,D)$ admits a $G$-action if 
$G\subset \Aut(Y,D)$ is a finite subgroup of the automorphism group. 
\end{defn}

We consider a special  finite group $G$-action on a Looijenga pair. 
Let us now restrict to negative-definite Looijenga pairs $(Y,D)$.  Artin's criterion for contractibility implies that $D$ can be analytically contracted to a singular cusp point
$$\pi: (Y,D)\to (\overline{Y},p).$$

\begin{defn}\label{defn_pair_hyperbolic_type_G}
A finite group $G$-action on a negative-definite Looijenga pair is called {\em hyperbolic type}  if the $G$-action only has quotient isolated singularities on 
$Y-D$ up to $G$-birational contraction $Y\setminus D\to \overline{Y\setminus D}$ , and there exists an open neighborhood $V_D\subset Y$ of $D$ such that 
$V_D$ is isomorphic to the neighborhood 
$V_C$ constructed in \S \ref{subsec_G_Inoue} and the $G$-action on $V_D$ is induced by the action on $V_C$ in \S \ref{subsec_G_Inoue}. 
Moreover, the quotient space $(V_{D}-D)/G$ is isomorphic to another open analytic space 
$V_{E}$ in \S \ref{subsec_G_Inoue}, so that adding one cusp point $q$ to $(V_{D}-D)/G$ we get a neighborhood $\bbV_{E}$ of $q$. 
\end{defn}

\begin{rmk}
Definition \ref{defn_pair_hyperbolic_type_G} implies that if a $G$-action on $(Y,D)$ is hyperbolic, then the quotient 
$(Y/D)/G=(X,E)$, possibly after resolution of singularities,  is also a Looijenga pair. 
\end{rmk}

\begin{rmk}
In general it is interesting to study the symplectic finite group $G$ action on a pair $(Y, \omega)$, where $Y$ is a rational surface, and $\omega$ is a symplectic form. 
The quotient  of a rational surface by a finite group is unirational, hence by Castelnuovo's theorem, is still rational. 
\end{rmk}

It is interesting to study the hyperbolic action more. 

\begin{defn}\label{defn_lattices}
    Let $(Y,D)$ be a negative definite  Looijenga pair.  Define 
    $$\Lambda(Y,D):=\{[D_0],\cdots, [D_{n-1}]\}^{\perp}\subseteq H^2(Y,\zz)$$
    i.e., the orthogonal complement of the lattice spanned by the classes $\{[D_i]\}$.
    We denote $L:=\{[D_0],\cdots, [D_{n-1}]\}$ the lattice spanned by $\{[D_i]\}$ in $H^2(Y,\zz)$.
\end{defn}
From \cite[\S3, Theorem 1]{Pinkham}, \cite[Lemma 1.5]{Friedman}, we have the rank of the lattice $\Lambda(Y,D)$
$$\rk(\Lambda(Y,D))=10+n-s=Q(Y,D)-2+r$$
where $s$ is the length of the resolution cycle of the dual cusp to the cusp $D$, and 
$r$ is the rank of the kernel 
$$\oplus_i\zz[D_i]\to H^2(Y,\zz).$$

We have:

\begin{thm}\label{thm_lattice_D_Lambda}
    Let $(Y,D)$ be a negative definite Looijenga pair with minimal period $d=(d_0,\cdots, d_{n-1})$. Suppose that there is a hyperbolic action of a finite group $G$ on $(Y,D)$. We have
    \begin{enumerate}
        \item  suppose that $L\to H^2(Y;\zz)$ is primitive, then $\Lambda(Y,D)^*/\Lambda(Y,D)\cong T$, where $T=T(M,\sigma)$ is the torsion subgroup of $H_1(\Sigma,\zz)_{\tor}$, where $\Sigma$ is the link of the cusp determined by the cycle $D$ in \S \ref{subsec_Hirzebruch}. Here $M$ is the rank two lattice and we have 
        $T=M/(\sigma-1)M$.
        \item The action $G$ fixes $L$ and move the lattice $\Lambda(Y,D)$. 
    \end{enumerate}
\end{thm}
\begin{proof}
    Item (1) is just from \cite[\S 3, Theorem 1]{Pinkham}.
 Item (2) is from the hyperbolic action, since $G$ acts on a neighborhood $V_D$ of $D$ in the surface $Y$.
\end{proof}

\begin{prop}\label{prop_pair_G_hyperbolic}
    Let $(Y,D)$ be a negative definite Looijenga pair with minimal period $d=(d_0,\cdots, d_{n-1})$, together with a hyperbolic action of a finite group $G$.
    Then the $G$-minimal model of $(Y,D)$ must be the case (1), (2) in 
    Proposition \ref{prop_G-minimal_pair_cusp}.
\end{prop}
\begin{proof}
We take $(Y,D)$ as a $G$-pair, then it is in the case of Proposition \ref{prop_G-minimal_pair_cusp}.  In the hyperbolic $G$-action case, the $G$-birational contraction of $G$-stable divisors  along the divisor $D$ depends on the $G$-action.  
\end{proof}

\begin{rmk}
     In \cite{FE}, Friedman-Engel used $\Lambda(Y,D)$ and the monodromy action  to define monodromy invariants for the smoothing of type III anti-canonical pairs.
     \cite{Jiang_2025} generalized it to the equivariant setting studied in this paper. 
\end{rmk}

\subsection{Toric model of $G$-Looijenga pairs}

In this section  we generalize the process of internal blow-ups and corner blow-ups in the equivariant setting using the main result in \cite[Theorem 0.1]{AW}. We prove that for any hyperbolic finite group $G$ action on a smooth Looijenga pair, there exists a $G$-toric model.  Note that in \cite[Corollary 2.12]{Trepalin} for a smooth rational surface with a $G$ action, there is a $G$-toric model. 

\begin{thm}\label{thm_Abramovich_Wang}(\cite[Theorem 0.1]{AW})
    Let $X$ be a projective variety of finite type over $\cc$, and let $Z\subset X$ be a proper closed subset. Let $G\subset\Aut_{\cc}(Z\subset X)$ be a finite group. Then there is a $G$-equivariant modification $r:X_1\to  X$ such that $X_1$ is nonsingular projective variety, and $r^{-1}(Z_{\red})$ is a $G$-strict divisor of normal crossings.
\end{thm}

The method of proof is to use the $G$-equivariant blowing-up. 
We have the following result which is interesting in its own right. 

\begin{thm}\label{thm_corner_internal_blow-up_G-toric_model}
Let $(Y,D)$ be a smooth  Looijenga pair endowed with a finite group $G$-action. 
We require that $G$ preserves $D$, i.e., $G$ acts on a neighborhood of $D$ in $Y$; and a $G$-orbit of $(-1)$-curves intersecting with the interior of the components of $D$ are disjoint. 
Then we can extend the $G$-action to the corner blow-ups 
and internal blow-ups to get the toric model 
$(Y^{\toric}, D^{\toric})$.  Moreover, $(Y^{\toric}, D^{\toric})$ also admits an action of $G$ such that the quotient $(Y^{\toric}, D^{\toric})/G$, up to resolution of singularities, is the toric model of the quotient $(Y,D)/G=(X,E)$. 
\end{thm}
\begin{proof}
We start from the quotient $(X,E)=(Y,D)/G$. From Proposition \ref{prop_G-minimal_pair_cusp}, there are two cases based on the Picard rank of $\Pic(Y)^G$. In any case the quotient $X=Y/G$ is a rational surface with quotient singularities.   By performing resolution of singularities and a choice $E\in |-K_X|$ we have a (probably singular) Looijenga pair $(X,E)$. 
We perform corner blow-ups, and then internal blow-ups to get the toric model of $(X,E)$
$$(X^{\toric}, E^{\toric})\stackrel{\text{internal blow-ups}}{\longleftarrow}  (\widetilde{X}, \widetilde{E})\stackrel{\text{corner blow-ups}}{\longrightarrow} (X,E).$$
This can be done since after several corner blow-ups we can use the result in \cite{Engel}, \cite[Proposition 1.3]{GHK15}.

Next we argue that the corner blow-ups and internal blow-ups can be lifted to $(Y,D)$ under the $G$-action. 
From the assumption  there is a neighborhood $V_D\subset Y$ such that $G$ preserves $V_D$, then the finite group $G$ lies in $\Aut_{\cc}(D\subset Y)$.  
We first look at the corner blow-ups. 
The corner blow-ups $(\widetilde{X}, \widetilde{E})\to (X,E)$ can be lifted to corner blow-ups $(\widetilde{Y}, \widetilde{D})\to (Y,D)$ $G$-equivariantly. 
Let $p\in D_i\cap D_{i+1}$ be a corner in $D$, from the action of $G$ on the curves $D$ and $V_D$ in 
\S \ref{subsec_G_Inoue}, $G$ fixes the corner $p$, so from 
\cite[Theorem 0.1]{AW},  there is a $G$-equivariant corner blow-up
$$(Y_p, D_p)\to (Y,D)$$
along $p$.  This morphism $(Y_p, D_p)\to (Y,D)$ induces a corner blow-up for $(X,E)$.  
We can perform this process to get a Looijenga pair $(\widetilde{Y}, \widetilde{D})$ with the desired  number of components 
$\widetilde{D}_i$ in $\widetilde{D}$ we want.  From Proposition \ref{prop_internal_corner_blow-up_charge},  corner blow-ups do not change the charge $Q(Y, D)$. 

For the Looijenga pair  $(\widetilde{Y}, \widetilde{D})$, we need to do the internal blow-downs to get the toric model 
$(Y^{\toric}, D^{\toric})$, and the blow-down maps are $G$-equivariant. 
From Proposition \ref{prop_internal_corner_blow-up_charge}, each time internal blow-down along a component $D_i\subset D$ changes the 
negative self-intersection number $d_i$ to $d_i-1$.  

From the $G$-action on the neighborhood $V_D$, $G$ only fix the corners of $D$,   so  there are no 
$G$-invariant exceptional $\pp^1$'s intersecting with the corners in $D$. Thus, a $G$-orbit of the exceptional $\pp^1$'s  intersecting with some interior  $D_j$'s in $D$ must contain the same number of $\pp^1$'s. They are disjoint since the action $G$ is almost  free on $\widetilde{Y}\setminus \widetilde{D}$. 
Therefore, a $G$-stable exceptional divisor contains a $G$-orbit of exceptional curves. 
Then we do  certain 
$G$-equivariant internal blow-downs of the $G$-orbits and  get a Looijenga pair $(Y^{\toric}, D^{\toric})$ with 
$Q(Y^{\toric}, D^{\toric})=0$.  From Proposition \ref{prop_toric_model}, the Looijenga pair  $(Y^{\toric}, D^{\toric})$ 
is toric, which is the toric model we want. 
We represent the above construction in  the following diagram:
\[
\xymatrixcolsep{6pc}\xymatrix{
(Y^{\toric}, D^{\toric})\ar[d]_{\pi}\ar@{<-}[r]^-{\text{internal blow-ups}} &  (\widetilde{Y}, \widetilde{D})\ar[d]_{\pi}\ar[r]^{\text{corner blow-ups}}& (Y,D)\ar[d]_{\pi}\\
(X^{\toric}, E^{\toric})\ar@{<-}[r]^-{\text{internal blow-ups}}&  (\widetilde{X}, \widetilde{E})\ar[r]^{\text{corner blow-ups}}& (X,E).
}
\]
Note the quotient $(Y^{\toric}, D^{\toric})/G$ is not exactly $(X^{\toric}, E^{\toric})$, but we can obtain $(X^{\toric}, E^{\toric})$ by resolution of singularities. 
\end{proof}

\begin{example}
We look at one example in the proof of  Proposition \ref{prop_G-minimal_pair_cusp}. 
Let $Y=Bl_{3}\pp^2$ be the degree $6$ del Pezzo surface.  $D\in |-K_Y|$ contains $6$ $(-1)$-curves intersecting transversally. Thus, $(Y,D)$ is a Looijenga pair.

 The group $\mu_6=\langle g\rangle$ acts on $Y$ and $Y^{\mu_6}$ has three fixed points with type $\frac{1}{2}(1,1), \frac{1}{3}(1,2)$, and $\frac{1}{6}(1,5)$ respectively. The quotient $X=Y/\mu_6$ is a Gorenstein log del Pezzo surface with three types $\frac{1}{2}(1,1), \frac{1}{3}(1,2)$, and $\frac{1}{6}(1,5)$ singularities.
The quotient $E=D/\mu_6$ is a rational nodal curve with self-intersection number $-1$.  We have $K_X^2=1$ and $E^2=-1$.
The quotient $(X,E)$ is not toric  since the charge $Q(X,E)=12+2+1-3=12$.  Therefore, we need to perform corner blow-ups along $E$, and then $12$-internal blow-downs to get the toric model. 

We first perform two corner blow-ups along the rational nodal curve $E$, and get 
$(X_1, E_1)$ so that $E_1$ contains four components with negative self-intersection sequence 
$(6, 2,1,3)$. Then performing $12$ times internal blow-downs  get the toric model 
$(X^{\toric}, E^{\toric})$ with $E^{\toric}$ the anti-canonical divisor whose negative self-intersection sequence is  given by $(0,-2, 0,2)$.   This is a minimal toric pair and 
$X^{\toric}$ is the degree $2$ Hirzebruch surface. 
\end{example}

\begin{prop}\label{prop_corner_internal_blow-up_G}
Let $(Y,D)$ be a negative  definite Looijenga pair endowed with a finite group $G$-action. 
Suppose that the action is hyperbolic.    Then we can extend the $G$-action to the corner blow-ups 
and internal blow-ups, such that the toric model 
$(Y^{\toric}, D^{\toric})$ also admits an action of $G$ such that the quotient $(Y^{\toric}, D^{\toric})/G$, up to resolution of singularities, gives the toric model of the 
quotient $(Y,D)/G=(X,E)$. 
\end{prop}
\begin{proof}
The action of the finite group $G$ on the Looijenga pair  $(Y,D)$ is hyperbolic,  which means that the  $G$-action   on $Y-D$ has only isolated quotient singularities.  The pair $(Y, D)$ is negative definite,  so  the negative self-intersection  numbers 
$d_i\ge 2$ for any $i$, and some $d_j\ge 3$.    Therefore, we can contract the divisor $D$ to a cusp $p\in Y^\prime$.  Therefore, the action of $G$ on $(Y,D)$ satisfies the conditions in Theorem \ref{thm_corner_internal_blow-up_G-toric_model}.  Thus, we are done. 

In order to remember the hyperbolic action, we restate the proof in this case. 
We know that $G$ acts on the neighborhood $V_D$ of $D$ the defined action in (\ref{eqn_action1}) and (\ref{eqn_action2}). The finite group $G$ contains two parts, one is in $\UU_M/\sigma$ and the other is in $\overline{M}/M$ defined in \S \ref{subsec_G_Inoue}. If the sequence $(d_0,\cdots, d_{n-1})$ is not a minimal period, then $\UU_M/\sigma$ acts by rotating the minimal period $(d_0,\cdots, d_k)$, and the $\overline{M}/M$ just acts on $V_D$ as in (\ref{eqn_action1}) and (\ref{eqn_action2}).
We perform corner blow-ups, and then internal blow-ups to get the toric model of $(X,E)$
$$(X^{\toric}, E^{\toric})\stackrel{\text{internal blow-ups}}{\longleftarrow}  (\widetilde{X}, \widetilde{E})\stackrel{\text{corner blow-ups}}{\longrightarrow} (X,E).$$

Next we argue that the corner blow-ups and internal blow-ups can be lifted to $(Y,D)$ under the $G$-action. 
Since there is a neighborhood $V_D\subset Y$ such that $G$ preserves $V_D$, then the finite group $G$ lies in $\Aut_{\cc}(D\subset Y)$.  
We first look at the corner blow-ups. 
The corner blow-ups $(\widetilde{X}, \widetilde{E})\to (X,E)$ can be lifted to corner blow-ups $(\widetilde{Y}, \widetilde{D})\to (Y,D)$ $G$-equivariantly. 
Let $p\in D_i\cap D_{i+1}$ be a corner in $D$, from the action of $G$ on the curves $D$ and $V_D$ in 
\S \ref{subsec_G_Inoue}, $G$ fixes the corner $p$, so from 
\cite[Theorem 0.1]{AW},  there is a $G$-equivariant corner blow-up
$$(Y_p, D_p)\to (Y,D)$$
along $p$, and from \cite[Lemma 2.2]{AW}, the $G$-action on the exceptional $\pp^1$ is similar to (\ref{eqn_action2}).  This morphism $(Y_p, D_p)\to (Y,D)$ induces a corner blow-up for $(X,E)$.  If the sequence $(d_0,\cdots, d_{n-1})$ is a minimal period, then it is just from \cite[Lemma 2.2]{AW}.  If the sequence $(d_0,\cdots, d_{n-1})$ is not a minimal period, then $G$ rotates the minimal period $(d_0,\cdots, d_k)$, and we have the same number  (corresponding to the cyclic order of $G$ in the part of $\UU_M/\sigma$) of corners to the same corner in $E$. The corner blow-up of $(X,E)$ still lifts to $(Y,D)$, but do the corner blow-ups on $(Y,D)$ for a $G$ orbit. 
We can perform this process to get a Looijenga pair $(\widetilde{Y}, \widetilde{D})$ with the desired  number of components 
$\widetilde{D}_i$ in $\widetilde{D}$ we want.  From Proposition \ref{prop_internal_corner_blow-up_charge},  corner blow-ups do not change the charge $Q(Y, D)$. 

For the Looijenga pair  $(\widetilde{Y}, \widetilde{D})$, we need to do the internal blow-downs to get the toric model 
$(Y^{\toric}, D^{\toric})$, and the blow-down maps are $G$-equivariant. 
From Proposition \ref{prop_internal_corner_blow-up_charge}, each time internal blow-down along a component $D_i\subset D$ changes the 
negative self-intersection number $d_i$ to $d_i-1$.  

From the $G$-action on the neighborhood $V_D$, $G$ only fix the corners of $D$.  So  there are no 
$G$-invariant exceptional $\pp^1$'s in $\widetilde{Y}$, and a $G$-orbit of the exceptional $\pp^1$'s  intersecting with some interior  $D_j$'s in $D$ must contain the same number of $\pp^1$'s. They are disjoint since the action $G$ is free on $\widetilde{Y}\setminus \widetilde{D}$ and on $\widetilde{D}$ except some corners. 
Then we do  certain 
$G$-equivariant internal blow-downs and  get a Looijenga pair $(Y^{\toric}, D^{\toric})$ with 
$Q(Y^{\toric}, D^{\toric})=0$.  
\end{proof}

 \section{Equivariant Type III canonical degenerations}\label{sec_typeIII-degeneration}

 In this section we generalize the Type III degeneration of anticanonical pairs in \cite{Engel}, \cite{FE} to the equivariant setting. 

\subsection{Universal deformation of Inoue-Hirzebruch surfaces}\label{subsec_universal_deformation_G}

Let $(\bbV,p,p^\prime)$ be the Inoue-Hirzebruch surface with two dual cusp singularities. 
Looijenga \cite[III Corollary 2.3]{Looijenga} proved that the surface $\bbV$ admits a universal deformation.  In particular,  Looijenga proved that $\bbV$ admits  smoothings.   Suppose that there is a finite group 
$G$ action on the Inoue-Hirzebruch surface $\bbV$.  The proof in \cite[II \S 2]{Looijenga} works in the $G$-equivariant case, therefore implies that $\bbV$ admits 
a universal $G$-deformation. 

Let us assume that 
$$\bbsV\to \Delta$$
is a $G$-equivariant smoothing of  $(\bbV,p,p^\prime)$ along the cusp $p^\prime$. 
So $\bbsV_0=\bbV$, and the cusp $p$ stays constant.  Any fiber $\bbsV_t (t\neq 0)$ is a surface with a  cusp singularity $p=p_t$ and possible ADE singularities. 

We resolve $p_t$ in the family under the group $G$-action and get a family
$$\pi: \sY\to \Delta$$
such that 
$\sY_0=\bV_0$ and $(\bV_0, p^\prime)$ is the partially contracted Inoue-Hirzebruch surface from $(V, D, D^\prime)$ with only cusp singularity $p^\prime$. 
This $G$-equivariant resolution exists, since on each individual fiber of $\bbsV\to \Delta$, the minimal resolution is $G$-equivariant by the
analytic-local description of the cusp neighborhood in \S \ref{sec_Inoue-Hirzebruch}. So,
the simultaneous minimal resolution is automatically $G$-equivariant.
For $t\neq 0$, $\sY_t$ is a simply connected surface with an anti-canonical divisor $D\in |-K_Y|$.  Thus $\sY_t$ is  a rational surface with possible ADE singularities. 
Since we do a $G$-equivariant simultaneous resolution of the singularities $p_t$ of the family $\bbsV\to \Delta$, for each fiber $\sY_t$ in the resolution family $\sY\to \Delta$, there must 
exist a subgroup $H\subset G$ acting faithfully on the fiber $\sY_t$.  Our group $G$ acts originally on the surface $\bbV$ and  $\bbsV$. After taking the $G$-equivariant resolution, 
the group $H$ acts on the fiber $\sY_t$ which preserves the anti-canonical divisor $D_t\in |-K_{\sY_t}|$. 

\subsection{Type III equivariant degeneration}\label{subsec_type_III_equivariant}

We are ready to introduce the Type III degeneration pairs.   
Let $\pi: \sY\to \Delta$ be the $G$-equivariant family constructed above.
Here is the definition of Type III degeneration pairs in \cite{Engel}. 
\begin{equation}\label{eqn_XX_0_engel}
\XX_0=\bigcup_{i=0}^{f-1}V_i, 
\end{equation} 
where 
\begin{enumerate}
\item $V_0=(V, D, D^\prime)$ is the compact Inoue-Hirzebruch surface. For $i>0$, the normalization $\widetilde{V}_i$ of $V_i$ is a smooth rational surface.  
\item We let $D_{ij}$ be the irreducible double curve of $\XX_0$ lying on $V_i$ and $V_j$ (in the case $V_i$ is not normal, we may have $i=j$). 
Let $D_i=\cup D_{ij}\subset V_i$ and $\widetilde{D}_i=\pi^{-1}(D_i)$ under $\pi: \widetilde{V}_i\to V_i$. Then $(\widetilde{V}_i, \widetilde{D}_i)$ is a Looijenga pair.  For 
$i=0$, $D_0=D^\prime$. 
\item  (Triple point formula) 
For the double curve $D_{ij}$ above, 
$$(D_{ij}|_{\widetilde{V}_i})^2+(D_{ij}|_{\widetilde{V}_j})^2=
\begin{cases}
-2, &  D_{ij} \text{~are smooth};\\
0, &  D_{ij} \text{~are nodal}.
\end{cases}
$$
\item The dual complex $\Gamma(\XX_0)$ of  $\XX_0$ is a triangulation of sphere.  
\end{enumerate}
From Friedman-Miranda \cite{FM}, $(\XX_0, D)$ admits a smoothing $\pi: (\sY,D)\to \Delta$ which is d-semistable. 

\begin{rmk}
    It is useful to recall the dual complex $\Gamma(\XX_0)$ here.  There are three data:
    \begin{enumerate}
        \item The vertices of $\Gamma(\XX_0)$ are given by $\{v_0, \cdots, v_{f-1}\}$ corresponding to each component $V_i$ in $\XX_0$.
        \item The edges $e_{ij}=(v_i, v_j)$ correspond to the double curves $D_{ij}$.
        \item The faces (triangles) $f_{ijk}=(v_i, v_j, v_k)$ correspond to triple points. 
    \end{enumerate}
    Each triangular face $f_{ijk}$ is integral-affine equivalent to a basis triangle, i.e., a lattice triangle of area $1/2$.  Therefore, integral-affine structures on the faces glue to give the integral-affine surface $\Gamma(\XX_0)$. 
\end{rmk}

The generic fiber of $\pi$ above is $(Y,D)$ the Looijenga pair.   The $G$-action on  $(Y,D)$ is hyperbolic, which means that $G$ acts on a neighborhood of $V_D\subset Y$ of $D$ as in  \S \ref{subsec_G_Inoue} so that the quotient is still a cusp. 
Our goal is to make the construction of Friedman-Miranda \cite{FM} for the Type III canonical degeneration pair work in the  $G$-equivariant  setting.  We construct the $G$-action on $\XX_0$. 
\begin{construction}\label{eqn_construction1}
    The $G$-action on $\XX_0$ is defined as follows.  Recall that the group $G\subset \Aut(V_0)$ acts on the Inoue-Hirzebruch surface in \S \ref{subsec_G_Inoue}. 
For any $D_{0j}\subset D^\prime=D_0=\cup_{i}D_{0i}$, whenever $g\in G$ acts on $D_{0j}$ as in  \S \ref{subsec_G_Inoue}, then $g$ acts on $D_{j0}$ the same way.   

If $d=(d^\prime_0,\cdots, d_{s-1}^\prime)$ is a minimal period, then the $\UU_M/\sigma$ part of the  $G$-action is trivial, and the $\overline{M}/M$ part acts on $D^\prime$ by the form  in (\ref{eqn_action1}) and (\ref{eqn_action2}).  Thus, the $G$-action on $V_i$ only acts on the $D_
i$ connecting with $D^\prime=D_0$. 

If $d=(d^\prime_0,\cdots, d_{s-1}^\prime)$ is not a minimal period, then the $\UU_M/\sigma$ part of the  $G$-action rotates the minimal period $(d_0^\prime,\cdots, d_{l-1}^\prime)$, and the $\overline{M}/M$ part acts on $D^\prime$ by the form  in (\ref{eqn_action1}) and (\ref{eqn_action2}). Then in this case
the $G$-action on $V_i$ for $i>0$ permutes the components $V_i$ such that it is compatible with the action of the cyclic part $\UU_M/\sigma$ if $G$ acts on $D^\prime=D_0$.

In any case, we require that $\overline{\XX}_0=\XX_0/G$ is also a Type III degeneration pairs with $\overline{V}_0=V_0/G$ the quotient of the Inoue surface $V$. From the hyperbolic action, $\overline{V}_0=V_0/G$ contains fixed points on the corners of the resolution cycles of the two cusps.  Taking suitable resolution of singularities we have an Inoue surface corresponding to the quotient cusps under the $G$-action. Of course, when taking resolution of singularities along the $G$ fixed points, the other components $V_i$ do the same process of resolutions accordingly. 
\end{construction}

\begin{defn}\label{defn_typeIII_G}
We construct the following $G$-equivariant Type III degeneration pairs
\begin{equation}\label{eqn_XX_0}
\XX_0=\bigcup_{i=0}^{f-1}V_i, 
\end{equation} 
where 
\begin{enumerate}
\item $V_0=(V, D, D^\prime)$ is the compact Inoue-Hirzebruch surface which admits a $G$-action as in \S \ref{subsec_G_Inoue}, such that the quotient $V_0/G$ is, after suitable resolution of singularities,  another Inoue-Hirzebruch surface $(W, E, E^\prime)$. .  For $i>0$, the normalization $\widetilde{V}_i$ of $V_i$ is a smooth rational surface. 
\item We let $D_{ij}$ be the irreducible double curve of $\XX_0$ lying on $V_i$ and $V_j$ (in the case $V_i$ is not normal, we may have $i=j$). 
Let $D_i=\cup D_{ij}\subset V_i$ and $\widetilde{D}_i=\pi^{-1}(D_i)$ under $\pi: \widetilde{V}_i\to V_i$. Then $(\widetilde{V}_i, \widetilde{D}_i)$ is a Looijenga pair.  For 
$i=0$, $D_0=D^\prime$.  
\item The $G$-action on $\XX_0$ is from Construction \ref{eqn_construction1}.
\item  (Triple point formula) 
For the double curve $D_{ij}$ above, 
$$(D_{ij}|_{\widetilde{V}_i})^2+(D_{ij}|_{\widetilde{V}_j})^2=
\begin{cases}
-2, &  D_{ij} \text{~are smooth};\\
0, &  D_{ij} \text{~are nodal}.
\end{cases}
$$

\item The dual complex $\Gamma(\XX_0)$ of  $\XX_0$ is a triangulation of sphere.  Furthermore, $\Gamma(\XX_0)$ admits a $G$-action such that 
$\Gamma(\overline{\XX}_0)=\Gamma(\XX_0)/G$ is also a triangulation of sphere.
There are two cases for the action of the dual complex.  If $d=(d^\prime_0,\cdots, d_{s-1}^\prime)$ is a minimal period, then the $G$-action on the dual complex is topologically trivial.
If $d=(d^\prime_0,\cdots, d_{s-1}^\prime)$ is not a minimal period, then the $G$-action on the dual complex $\Gamma(\XX_0)$ is not trivial, it will permute the components $V_i$ connecting with $D^\prime$ in $V_0$. In this case all the other components $V_j$'s not connecting with $V_0$ permute in a compatible way. 
\end{enumerate}
\end{defn}

\begin{rmk}
    From the action of $G$ on $\XX_0$, the quotient $\XX_0/G=\overline{\XX}_0$ is a Type III degeneration pairs with quotient singularities.  We take the resolution of singularities along the corners of $D^\prime$ and $D$, at the same time add new components $V_k$'s corresponding to the new components produced by the corner blow-ups, we get a Type III degeneration pair $\widetilde{\overline{\XX}}_0=\bigcup_{i=0}^{r-1}\tilde{V}_i$ such that $\tilde{V}_0=(W,E, E^\prime)$.

    From \cite[Theorem 2.26]{FE}, the deformation functor of $\widetilde{\overline{\XX}}_0$, keeping the divisor $E$ with normal crossings, has tangent space $\Ext^1(\Omega_{\widetilde{\overline{\XX}}_0}(\log E), \sO_{\widetilde{\overline{\XX}}_0})$.   \cite[Theorem 2.26]{FE} proved that there exists a unique smoothing component $(M,0)$ of $(\widetilde{\overline{\XX}}_0, E)$, and the discriminant locus in $M$ is a smooth hypersurfacce.  Thus, our  quotient pair $(\overline{\XX}_0, D/G)$ lies in the discriminant locus, and can be deformed to $(\widetilde{\overline{\XX}}_0, E)$. Up to deformation,  we can use the dual complex $\Gamma(\widetilde{\overline{\XX}}_0)$ to replace with $\Gamma(\overline{\XX}_0)$.
\end{rmk}

\begin{thm}\label{thm_smoothing_XX_0}
There exists a $G$-equivariant  smoothing  family $\pi: \XX\to \Delta$
such that $\sD\in |-K_{\XX}|$,   $\sD_t=D_t\in Y_t$ for a rational surface  $Y_t$ when $t\neq 0$,  and $\XX_0=\pi^{-1}(0)$ is the variety in (\ref{eqn_XX_0}) with $\sD_0=D_0$. 
\end{thm}
\begin{proof}
We prove the theorem by first generalizing  \cite[Lemma 2.9]{FM}.  Let $T_{\XX_0}^0$ and $T_{\XX_0}^1$ be the tangent sheaves so that
$$T_{\XX_0}^i=\sE xt^i(\Omega_{\XX_0}^1, \sO_{\XX_0}).$$
The the global tangent spaces are defined by
$$\T_{\XX_0}^i=\Ext^i(\Omega_{\XX_0}^1, \sO_{\XX_0}).$$
Recall that the variety $\XX_0$ is called $d$-semi-stable, if  
$T_{\XX_0}^1=\sO_{Q}$, where $Q\subset \XX_0$ is the singular locus.  We have that 

\begin{lem}
There always exists a construction  $\XX_0$ in (\ref{eqn_XX_0}) with a finite group  $G$-action such that 
 it is $d$-semi-stable. 
\end{lem}
\begin{proof}
We generalize \cite[Proposition (5.14)]{Friedman} in this setting by taking care of the $G$-action. 
Recall that $D_{ij}$ is the double curve in $V_i$ and $V_j$.  $D_i=\cup D_{ij}$ and we set 
$E:=\cup D_{i}$.   We let 
$D_{ij}^{0}:=D_{ij}-T$, where $T$ is the triple point locus. 
We know that $D_{ij}$ is smooth ($D_{ij}$ is not a nodal rational curve since $n>1$).
We show that there exists a choice of isomorphisms
$$\varphi_{ij}: D_{ij}^{0}\subset V_i\stackrel{\sim}{\rightarrow} D_{ij}^{0}\subset V_j$$
where the extension  $\overline{\varphi}_{ij}$ of $\varphi_{ij}$ to $D_{ij}$ fixes the triple points and the surface 
$\XX_0$ is $d$-semi-stable by the gluing of $\overline{\varphi}_{ij}$. 
The triple point formula implies that 
$(D_{ij}|_{\widetilde{V}_i})^2$ or $(D_{ij}|_{\widetilde{V}_j})^2$ is nonzero.  Our finite group  $G$ acts on $\XX_0$, and the surface $\XX_0=\cup_{i=0}^{n}V_i$,  where $V_0=(V, D, D^\prime)$ is an Inoue-Hirzebruch surface with a $G$-action.  For $i>0$, each 
$\tilde{V}_i\to V_i$ is a rational surface. 
We follow the same proof as in \cite[(5.14)]{FM}. Let 
$$G_{ij}=\{\text{divisors of degree zero on~} D_{ij}^{0}\}/\div(f)$$
where the $f$ are functions on $D_{ij}$, which are not zero or $\infty$ at the triple points $t_1, t_2$, and 
$f(t_1)=f(t_2)$.  Then we have that 
$$G_{ij}\cong \cc^*\subset \Pic^0(D_{ij})\subset \Pic^0(E).$$
Let $\widetilde{E}\to E$ be the normalization and consider the following exact sequence
\begin{equation}\label{eqn_exact_sequence}
0\to H^0(\sO_E^*)\rightarrow H^0(\sO_{\widetilde{E}}^*)\rightarrow H^0(\sO_{\widetilde{E}}^*/\sO_E^*)\rightarrow H^1(\sO_E^*)\to 0,
\end{equation}
the $\Pic(E)$ is determined by the gluing from $H^0(\sO_{\widetilde{E}}^*/\sO_E^*)$. 
From \cite[Definition 1.9]{Friedman}, we have 
$$\sO_{D_i}(-\XX_0)=(I_{D_i}/I_{D_i}^2)\otimes_{\sO_{D_i}}(I_{V_i}/I_{V_i}J_{D_i})$$
and 
$$\sO_{E}(-\XX_0)=(I_{V_0}/I_{V_0}I_{E})\otimes_{\sO_{E}}(I_{V_1}/I_{V_1}I_{E})\otimes_{\sO_{E}}\cdots \otimes_{\sO_{E}}(I_{V_n}/I_{V_n}I_{E})$$
where 
$$
\begin{cases}
I_{D_i}= \text{ideal sheaf of~} D_i \text{~in~} V_i;\\
I_{V_i}= \text{ideal sheaf of~} V_i \text{~in~} \XX_0;\\
J_{D_i}= \text{ideal sheaf of~} D_i \text{~in~} \XX_0.
\end{cases}
$$
From \cite[Definition 1.13]{Friedman}, $\XX_0$ is $d$-semi-stable if $\sO_{E}(\XX_0)=\sO_{E}$, which is equivalent to $T_{\XX_0}^1=\sO_{Q}$.  The locally free sheaf 
$\sO_{E}(-\XX_0)$ is defined by the trivial bundles $\sO_{D_{ij}}$, plus the gluing defined by using 
$$z_iz_jz_k\in H^0(\sO_{D_{ij}}(-V_i-V_j-T))$$
as a local section generator.  

The finite group $G$ acts on the variety $\XX_0$,only on a neighborhood of $D$ and $D^\prime$.  We can modify the gluing along $D_{ij}$ by 
$\lambda\in \Aut^0(D_{ij}^{0})$ which is compatible with the action $G$ such that 
$\sO_{E}(-\XX_0)$  has the gluing data at a triple point 
$t_{ijk}$,
$$
\begin{cases}
z_iz_jz_k\in H^0(\sO_{D_{ik}}(-V_i-V_k-T));\\
z_iz_jz_k\in H^0(\sO_{D_{jk}}(-V_j-V_k-T));\\
\lambda^{-1} z_iz_jz_k\in H^0(\sO_{D_{ij}}(-V_i-V_j-T)).
\end{cases}
$$
At the triple point $t_{ijl}$, the formula is similar.  Now look at the exact sequence (\ref{eqn_exact_sequence}), and we have 
$$\left(\sO^*_{\widetilde{E}}/\sO_E^*\right)_{t_{ijk}}\cong (\cc^*)^3/\cc^*.$$
If $(\cc^*)^3$ has basis $(e_{ij}, e_{jk}, e_{ik})$ and the action $\cc^*$ is the diagonal subspace, then by the gluing, the effect on $\sO_{E}(\XX_0)$
is to multiply the $e_{ij}$ component at $t_{ijk}$ by $\lambda$ and the corresponding component at 
$t_{ijl}$ by $\lambda^{-1}$.  This is exactly the action of 
$G_{ij}$ on $\Pic^0(E)$, up to a power of $2$.  Thus, we have 
$\sO_{E}(\XX_0)=\sO_E$.
\end{proof}

For the $d$-semi-stable  $G$-variety $\XX_0$, let 
$$\pi: \widetilde{\XX}_0\to \XX_0$$
be the normalization.  Let 
$\widetilde{T}\to T$ and $\widetilde{Q}\to Q$ be the corresponding normalizations of the locus $T$ and $Q$. 
Since $\XX_0$ is a variety with normal crossings, \cite[(3.2), (3.3)]{Friedman} implies that there exists an intrinsically defined subsheaf 
$$\Lambda^1_{\XX_0}\subset \pi_* \Omega^1_{\widetilde{\XX}_0}(\log \widetilde{Q})$$
and a resolution
$$0\to \Omega_{\XX_0}^1/\tau_{\XX_0}\rightarrow \Lambda^1_{\XX_0}\rightarrow \pi_* \sO_{\widetilde{Q}}\rightarrow \pi_* \sO_{\widetilde{T}}\to 0$$
where $\widetilde{T}=T$, $\tau_{\XX_0}$ is the torsion point of $\Omega^1_{\XX_0}$.  
Here the sheaf $\Lambda^1_{\XX_0}$ is intrinsic such that 
$\Lambda^2 \Lambda^1_{\XX_0}\cong \omega_{\XX_0}$.  Choose a generating section $\xi\in H^0(T^1_{\XX_0})$, and via Lie bracket, we have the map
\begin{equation}\label{eqn_Lie}
[\cdot, \xi]: T^0_{\XX_0}\to T^1_{\XX_0}.
\end{equation}
We have that 
$$S_{\XX_0}:=\ker([\cdot, \xi])\cong (\Lambda^1_{\XX_0})^*.$$

Now we use the  same proof in \cite[Lemma 2.7]{FM} and consider the $G$-equivariant setting of sheaves.   
We show that there exists smoothing of $\XX_0$ from  \cite[Lemma 2.7]{FM}, and then show in the end that a $G$-equivariant smoothing also exists. 
We first have   
$H^0(\XX_0, \Lambda^1_{\XX_0})=0$.  We have the following results as in \cite[Lemma 2.8]{FM}:
\begin{enumerate}
\item $H^2(T^0_{\XX_0})=0$;
\item The natural map $\T^1_{\XX_0}\to H^0(T_{\XX_0}^1)$ is surjective;
\item  The natural map $H^1(T_{\XX_0}^0)\otimes H^0(T_{\XX_0}^1)\to H^1(T_{\XX_0}^1)$ is surjective. 
\end{enumerate}
The first one is from the following resolution:
$$0\to \Omega_{\XX_0}^1/\tau_{\XX_0}\rightarrow  \pi_* \Omega_{\widetilde{\XX_0}}^1\rightarrow \pi_* \Omega_{\widetilde{Q}}^1\to 0.$$
We have $H^0(V_0, \Omega^1_{V_0})=0$, see \cite[(1.5.3)]{FM},  which implies that 
$H^0(\Omega_{\XX_0}^1/\tau_{\XX_0})=0$.  Serre duality implies that 
$$H^2(T_{\XX_0}^0)\cong H^0(\Omega_{\XX_0}^1/\tau_{\XX_0}\otimes \omega_{\XX_0})^*.$$
By the construction for $\XX_0$, we have 
$$\pi^*\omega_{\XX_0}|_{V_0}=\sO_{V_0}(-D)$$
and 
$$\pi^*\omega_{\XX_0}|_{V_i}=\sO_{V_i}, \quad i>0,$$
therefore,  $H^0(F\otimes \omega_{\XX_0})\subset H^0(F)$ for any torsion free coherent sheaf 
$F$. Thus, we get $H^2(T^0_{\XX_0})=0$ since $H^0(\Omega_{\XX_0}^1/\tau_{\XX_0})=0$.

(2)  comes from the $\Ext$ spectral sequence
$$\T_{\XX_0}^i=\bigoplus_{p+q=i}H^p(\XX_0, \sE xt^q(\Omega^1_{\XX_0}, \sO_{\XX_0})).$$

(3) is from (\ref{eqn_Lie}), since we have
$$0\to S_{\XX_0}\rightarrow T_{\XX_0}^0\stackrel{[\cdot, \xi]}{\rightarrow} T^1_{\XX_0}\to 0.$$
It is enough to show that $H^2(S_{\XX_0})=0$, or equivalently 
$H^0(\Lambda_{\XX_0}^1\otimes \omega_{\XX_0})=0$ which is true since 
$H^0(\Lambda_{\XX_0}^1)=0$.

Now it is ready to prove the theorem.  The proof is the same as 
in \cite[(5.10)]{Friedman}, but with the extra consideration of the $G$-action.  
First let 
$$\T^1_{\XX_0}\otimes \T^1_{\XX_0}\to \T^2_{\XX_0}$$
be the Lie bracket  map.  Since $H^2(T_{\XX_0}^0)=0$,  the Lie bracket $\T^1_{\XX_0}\otimes \T^1_{\XX_0}\to \T^2_{\XX_0}$
induces
$$[\cdot, \cdot]: H^1(T^0_{\XX_0})\otimes H^0(T^1_{\XX_0})\to H^1(T^1_{\XX_0}).$$
As in \cite[(5.10)]{FM}, we let 
$$W_1=H^0(T_{\XX_0}^1)\subset \T^1_{\XX_0},$$
a hyperplane since 
$\T_{\XX_0}^1=H^0(T^1_{\XX_0})\oplus H^1(T_{\XX_0}^0)$ and $H^0(T_{\XX_0}^1)\cong \cc$.  
Let $e\in \T_{\XX_0}^1$ be mapped to $1\in H^0(T_{\XX_0}^1)\cong \cc$, and 
$$W_2=\{v\in  \T_{\XX_0}^1| [v, e]=0\}=\{x+\lambda e| \lambda\in\cc, x\in W_1, [x,e]=0\}.$$
Then $W_1\cap W_2=\{x\in W_1| [x,e]=0\}$ is a hyperplane in $W_2$.  By the basic deformation theory, we have a holomorphic map:
$$f:  \T_{\XX_0}^1\to  \T_{\XX_0}^2$$
such that 
$f(0)=0$,  $f$ has no linear terms, and $f^{-1}(0)$ is the base space of a versal deformation of $\XX_0$. As in \cite[(5.10)]{Friedman}, $f^{-1}(0)$
contains the smooth divisor 
$N_1\subset W_1$ which corresponds to local trivial deformations. Then from \cite[(5.10)]{Friedman}, 
$$f^{-1}(0)=N_1\cup N_2$$
where $N_2=\{h(v)=0\}$ for 
$$h: (\T_{\XX_0}^1,0)\to  (\T_{\XX_0}^2,0)$$
such that $f=g\cdot h$, and 
$\{g=0\}$ is the reduced germ of $N_1$. 
Then $N_1$ corresponds to the local trivial deformations of $\XX_0$, and 
$N_2-N_1$ corresponds to smooth rational surface $(Y, D)$, which is from \cite[(2.5)]{Friedman}. 
Thus, the deformation theory implies that we have 
the surface $\XX_0$ admits a smoothing 
$\pi: \XX\to \Delta$.

To show that there exists a $G$-equivariant smoothing 
$\pi: \XX\to \Delta$,  it is sufficient to show that 
$$H^0(\XX_0, T^1_{\XX_0})^G\neq 0.$$
We have $T^1_{\XX_0}=\sO_{Q}$. 
Our $G$ acts on the type III degeneration $\XX_0$ as in (\ref{eqn_XX_0}). From the action if the negative self-intersection sequence $(d_0,\cdots, d_n)$ is a minimal period, then the $G$ acts on the divisor $D$ and $D^\prime=D_0$ as in (\ref{eqn_action1})
and (\ref{eqn_action2}). 
The divisor $D^\prime=D_0$ and $D_{ij}$ (double curves) are in the singular locus $Q$ of $\XX_0$. 
$G$ is finite and the action in (\ref{eqn_action1})
and (\ref{eqn_action2}) are coordinate wise on the $\pp^1$'s, so sufficient large power of the coordinate function on the $\pp^1$'s will be $G$-invariant.  Hence $H^0(\XX_0, T^1_{\XX_0})^G\neq 0$.

If the negative self-intersection sequence $(d_0,\cdots, d_n)$ is not a minimal period, then the $G$ action, except the above mentioned action on the divisors $D^\prime$, also rotates the minimal period of the negative self-intersection sequence $(d_0,\cdots, d_n)$.  But then the $G$-fixed part $\XX_0$ gives another type III degeneration $\overline{\XX}_0$, together with the $G$-action as the above case.  So $H^0(\XX_0, T^1_{\XX_0})^G\neq 0$. 
\end{proof}

\section{Construction of Type III  canonical degeneration pairs}\label{sec_typeIII_pair}

\subsection{Integral-affine surface}\label{subsec_integral-affine}

We recall the integral-affine surfaces in \cite{GHK15}, \cite[\S 3]{Engel}.

A basis triangle of $\rr^2$ is a triangle $\Delta$ of area $\frac{1}{2}$ with integral vertices in $\zz^2\subset \rr^2$.  Any two pairwise edges of a basis triangle form a basis for $\zz^2$. 

\begin{defn}\label{defn_triangle_integral-affine}(\cite[Definition 3.1]{Engel})
A triangulated integral-affine surface with singularities is a triangulated real surface $S$, possibly with boundary such that 
\begin{enumerate}
\item the complement of the vertices $\{v_i\}\subset S$ of the triangulation admits an atlas of charts into $\rr^2$, whose transition functions take values in 
$SL_2(\zz)\ltimes \zz^2$.
\item the interior of every triangle admits a chart to a basis triangle. 
\end{enumerate}
\end{defn}

An integral-affine surface with singularities has a canonical orientation induced from the standard orientation on $\rr^2$.  Let $e_{ij}$ be the edge $v_i-v_j$ in the triangulation of $S$.  Let $f_{ijk}$ be
the triangle whose counterclockwise ordered vertices are $v_i, v_j, v_k$. In this chart we can write $e_{ij}=v_j-v_i$.

Let $S$ be a triangulated real surface by basis triangles.  The boundary $\partial S=P_1+\cdots+P_n$ is a polygon, where each $P_i$ is integral-affine and is a line segment between two lattice points.
We assume that $\partial S$ is maximal which means the union of two distinct boundary components is never integral-affine equivalent to a single line segment.  

\begin{defn}\label{defn_singular_integral-affine}
If the atlas of integral-affine charts on $S-\{v_i\}$ extends to all vertices $\{v_i\}$, then we say $S$ is non-singular. 
Otherwise $S$ is singular.  Let $S_{\sing}$ denote the singular vertices, i.e., the vertices which the integral-affine structure fails to extend.  
\end{defn}

\begin{rmk}\label{rmk_triple_point}
Let $f_{ijk}$ be a triangle formed by $v_i, v_j, v_k$ in the counterclockwise direction. Let $v_i-v_l$ be another edge such that $v_i, v_k, v_l$ form another triangle $f_{ikl}$ in the counterclockwise direction again.  We define the 
self-intersection number $d_{ik}$ by
$$d_{ik}e_{ik}=e_{ij}+e_{il}.$$
From \cite[Proposition 3.6, Proposition 3.7]{Engel}, $d_{ik}+d_{ki}=2$ for every interior edges $e_{ik}$. 
Also a triangulated integral-affine surface $S$ is uniquely determined by the data of a collection of negative self-intersections $d_{ik}$ for each directed interior edge $e_{ik}$
such that $d_{ik}+d_{ki}=2$.
\end{rmk}

\begin{defn}\label{defn_pseudo-fan_pair}
Let $(Y,D)$ be a Looijenga pair.  The pseudo-fan of $(Y,D)$ is a triangulated integral-affine surface whose underlying surface $S_{(Y,D)}$ is the cone over the dual complex of $D$. 
\end{defn}

Let $e_i$ be the edge from the cone point to the vertex corresponding to $D_i$. Then the negative self-intersection of $e_i$ is:
$$
d_i=
\begin{cases}
-D_i^2, & n>1;\\
2-D_i^2, & n=1.
\end{cases}
$$
Also from \cite[Proposition 3.9]{Engel}, the integral-affine structure on the pseudo-fan of $(Y,D)$ extends to the cone point if and only if $(Y,D)$ is a toric pair. 
Here we recall that a toric Looijenga pair $(Y^{\toric}, D^{\toric})$ is a toric surface $Y^{\toric}$ such that $D^{\toric}$ is its toric boundary.

For a Type III canonical degeneration pair $\XX_0$, the dual complex $\Gamma(\XX_0)$ is a triangulation of the sphere $S^2$. 
The vertices $\{v_i\}$ correspond to the components $V_i$, the directed edges $e_{ij}$ correspond to double curves $D_{ij}$, and triangular faces $f_{ijk}$ correspond to 
triple points in Remark \ref{rmk_triple_point}.

From \cite[Proposition 3.10]{Engel}, the dual complex $\Gamma(\XX_0)$ has a triangulated integral-affine structure such that 
$$
d_{ij}:=
\begin{cases}
-D_{ij}^2, & \ell(D_i)\ge 2;\\
2-D_{ij}^2, & \ell(D_i)=1
\end{cases}
$$
where $d_{ij}$ is the negative self-intersection of $e_{ij}$.   Moreover, the integral-affine structure extends maximally to 
$\Gamma(\XX_0)-\left(\{v_i| Q(V_i, D_i)>0\}\cup\{v_0\}\right)$. It is easy to calculate that $d_{ij}+d_{ji}=2$.

\cite[Proposition 3.7]{Engel} proved that a  triangulated integral-affine surface is uniquely determined by the data of a collection of negative self-intersections $d_{ik}$ for each directed interior edge $e_{ik}$ such that $d_{ik} + d_{ki} = 2$. Therefore, $\Gamma(\XX_0)$ is uniquely determined by the data $\{d_{ij}\}$.  Also the this integral-affine surface $\Gamma(\XX_0)$ has only $A_1$-singularities.

\subsection{Surgeries}\label{subsec_surgeries}

Let us recall the surgeries on the integral-affine surface in \cite[\S 4]{Engel}. 
The surgeries on the integral-affine surface are motivated by the almost toric fibration in \cite{Symington}.  It is a generalization of the moment map from toric surfaces 
to its moment polygon $S$. 

Let $S$ be a singular integral-affine surface which is homeomorphic to a disc, and we let
$$\partial S=P_1+\cdots+P_n$$
is the union of a sequence of segments $P_i$ such that each segment integral-affine equivalent to a straight line segment between two lattice points. 
The boundary components $P_i$ go counterclockwise around $S$ when $i$ increases.  Denote by 
$$v_{i, i+1}=P_i\cap P_{i+1}$$
the vertex, and let $x_i, y_i$ be the primitive integral vectors emanating from $v_{i, i+1}$ along $P_{i+1}$ and $P_i$, respectively.  Then we have 
$y_{i+1}=-x_i$.  As in \cite[Definition 4.2]{Engel}, we define negative self-intersection $d_i$ of $P_i$ by:
$$d_i y_i=y_{i-1}-x_i=y_{i-1}+y_{i+1}.$$
If $\mu: (Y, D, \omega)\to S$ is an almost toric fibration, then  it is a Lagrangian fibration whose general fiber is a smooth $2$-torus, which degenerates under symplectic reduction, over the 
boundary $\partial S$. Also the interior fibers may also degenerate to necklaces of spheres at some finite set of points. 

There are two type of surgeries on $S$.  
\subsubsection{Internal blow-up}\label{subsubsec_internal}
The first one is the internal blow-up of $S$ on the boundary $P_i$.
The surgery is given by:

\textbf{Step I:}  Delete the triangle $T\subset S$ which satisfies the properties:
\begin{enumerate}
\item One edge $e_T$ of $T$ is proper subsegment of $P_i$;
\item $T\setminus e_{T}\subset S-S_{\sing}$ belongs to the interior part of $S-S_{\sing}$;
\item  $T$ is an integer multiple $n$ size of a basis triangle. 
\end{enumerate}

\textbf{Step II:}  Let $v$ be the unique vertex of $T$ lying in the interior of $S$, and let $(e_1, e_2)$ be the oriented lattice basis emanating from $v$ along 
the edges of $T$.  The glue the edge $e_2$ of $S-T$ to the edge along $e_1$ of $S-T$ via the unique affine-linear map which fixes $v$, and maps 
$e_2\mapsto e_1$, and preserving the line containing $P_i$.

The resulting integral-affine surface is an internal blow-up of $S$ on $P_i$.  The singular set is 
$S_{\sing}\cup\{v\}$ and $n$ is the size of the surgery.  Please see \cite[Figure 3]{Engel}. 

\subsubsection{Node-smoothing}\label{subsubsec_node}
The second  one is the node-smoothing of $S$ at the node $P_i\cap P_{i+1}$.
The surgery is:

At the node $P_i\cap P_{i+1}$, for $n\in\mathbb{N}$, cut a segment from $v_{i, i+1}$ to 
$$v:=v_{i,i+1}+n(x_i+y_i)$$
lying in $S-\partial S$.  Then we glue the clockwise edge of the cut to the counterclockwise edge of the cut by the shearing map which points to the line containing the cut and maps $x_i$ to $-y_i$. 

The resulting integral-affine surface is the smoothing of the node at $P_i\cap P_{i+1}$ and has size $n$. The singular set is 
$S_{\sing}\cup\{v\}$.

\subsubsection{Surgeries and self-intersection numbers}\label{subsubsec_node_internal}
Similar to Proposition \ref{prop_internal_corner_blow-up_charge}, an internal blow-up of the integral-affine surface $S$ on the boundary $P_i$ changes the negative self-intersections of the boundary components by:
$$(\cdots, d_i,\cdots)\mapsto (\cdots, d_i+1, \cdots).$$
A node smoothing at $P_i\cap P_{i+1}$ of $S$ changes the 
negative self-intersections of the boundary components by:
$$(\cdots, d_i, d_{i+1},\cdots)\mapsto (\cdots, d_i+d_{i+1}-2, \cdots).$$

Now suppose that we have an integral-affine disc such that the adjacent edges of $\partial S$ meet to form lattice bases, and the negative self-intersections of 
$P_i\subset \partial S$ are:
$$
\begin{cases}
d_i\ge 2, & \text{~for all~} i;\\
d_i\ge 3, & \text{~for some~} i.
\end{cases}
$$
Then from \cite[Proposition 4.6]{Engel}, there is a natural embedding $S\hookrightarrow \hat{S}$ where $\hat{S}$ is an integral-affine sphere and 
$\hat{S}_{\sing}=S_{\sing}\cup\{v_0\}$ for a distinguished point $v_0\in \hat{S}\setminus S$. 

From \cite[Remark 4.8,  Definition 4.7]{Engel}, $v_0\in \hat{S}$ may not be integral.  Since it is rational, we take the order $k$ refinement $S[k]$, and 
$\hat{S}[k]=S[k]\cup C[k]$, where $C:=\hat{S}\setminus S$.  Thus $v_0\in \hat{S}[k]$.

\subsection{The construction}\label{subsec_construction}

Now we are ready to construct a Type III canonical degeneration pair from a Looijenga pair $(Y,D)$, together with a finite group 
$G$-action.  We first have:

\begin{prop}\label{prop_Looijenga_internal_smoothing}
Let $(Y,D)$ be a  Looijenga pair $(Y,D)$ together with a hyperbolic $G$-action.  Then $(Y,D)$ can be represented by a sequence of $G$-equivariant node smoothings and $G$-equivariant internal blow-ups from a toric $G$-pair $(Y^{\toric}, D^{\toric})$. 
\end{prop}
\begin{proof}
Since the action of $G$ is hyperbolic, the $G$ action is almost free on $Y-D$.  
Then every Looijenga pair $(Y,D)$ can be expressed as a sequence of $G$-equivariant internal blow-ups and 
$G$-equivariant corner blow-ups from a $G$-minimal pair
$$
(Y, D)\stackrel{\alpha}{\rightarrow}  (\widetilde{Y}, \widetilde{D})\stackrel{\beta}{\rightarrow}(Y^{\min},D^{\min})$$
where $\beta$ contain corner blow-ups, and $\alpha$ contain internal blow-ups. 
For the node-smoothing in \S \ref{subsubsec_node},  geometrically it is like the smoothing of the singularity 
$(xy=0)$, and this process  can be made into $G$-equivariant since the $G$-action on the node is balanced, i.e., the actions on $x$ and $y$ are inverse to each other.

The  $G$-minimal pair $(Y^{\min},D^{\min})$ has its charge  $Q(Y^{\min},D^{\min})\ge 0$.  
From \S \ref{subsubsec_node_internal},  a node smoothing will increase the charge by $1$. 
Since the $G$-action  is balanced, the node smoothing can be made into $G$-equivariant. 
So 
every $G$-minimal Looijengal pair $(Y^{\min}, D^{\min})$ is a node
smoothing of a minimal toric  pair (which has charge zero).  
Thus,   $(Y^{\min},D^{\min})$ can be obtained by node smoothings from a toric pair $(Y^{\toric}, D^{\toric})$.
Thus,  as a pair, we know $(Y,D)$ can be given by a sequence of internal blow-ups and node-smoothing  from a toric pair. 
\end{proof}

Now suppose that we have a  Looijenga pair $(Y,D)$, together with a finite group $G$-action such that the $G$-action is hyperbolic on 
$(Y,D)$.    We also require that the cycle $D$ is negative-definite. 

We perform the arguments in \cite[\S 5]{Engel} to construct a Type III canonical degeneration pair $\XX_0=\cup_{i}V_i$, such that there exists a $G$-action on $V_0$. 
First for a Looijenga pair $(Y,D)$ with a $G$-action, from Proposition \ref{prop_Looijenga_internal_smoothing} there exists a sequence of $G$-equivariant internal blow-ups and corner blow-ups to a $G$-minimal pair
$$(Y,D)\to (Y^{\min}, D^{\min}),$$
and then the minimal pair has a toric model
$$ (Y^{\min}, D^{\min})\to (Y^{\toric}, D^{\toric}).$$
Let $S^{\toric}$ be the moment polygon of the toric model  $(Y^{\toric}, D^{\toric})$, then we perform the internal blow-ups and node smoothing as in \S \ref{subsec_surgeries} for 
$S^{\toric}$ to get the integral-affine surface $S$ for the Looijenga pair $(Y,D)$.  
Here the argument is the same as in \cite[\S 4]{Engel}. 
From the argument, there are totally $Q(Y,D)$ surgeries of fixed sizes. 

We can complete the  integral-affine surface $S$  to a sphere $\hat{S}$ as in \cite[Proposition 4.6]{Engel}.  We also take an order $k$-refinement
$\hat{S}[k]$ such that $v_0\in \hat{S}\setminus S$ is integral.  The refinement 
$\hat{S}[k]$ admits a triangulation into basis triangles. 

Note that there may exist many such triangulations and we choose the one that attains the minimal number of edges emanating from $v_0$. 

For each $v_i\in \hat{S}[k]$, if $v_i$ is non-singular, then $\Star(v_i)$ is the pseudo-fan of a toric surface pair $(V_i, D_i)$. 
Suppose that we have a vertex $v_i\in \hat{S}[k]_{\sing}$ which is singular, $v_i\neq v_0$.  Recall such a 
vertex $v_i$ is given by a surgery on $S^{\toric}$. Let 
$v_i^{\toric}\in S^{\toric}$ be the preimage of this vertex under the surgery.  Then we have

An internal blow-up on $S^{\toric}$ corresponds to a node smoothing on  $\Star(v^{\toric}_i)$.

A node smoothing on $S^{\toric}$ corresponds to an internal blow-up on  $\Star(v^{\toric}_i)$.
Please see \cite[Figure 8]{Engel} for the graph. 

Thus there exists a Looijenga pair $(V_i, D_i)$ with pseudo-fan $\Star(v_i)$. 
For $v_0\in \hat{S}[k]$, \cite[Lemma 5.3]{Engel} showed that $\Star(v_0)$ is the pseudo-fan of $(V_0, D^\prime)$.  
From the automorphism group explanation as in \S \ref{subsec_G_Inoue},  the finite group $G$ also acts on the Inoue-Hirzebruch surface 
$V_0$, so that these two dual cusps $D, D^\prime$ are all contractible.  This is exactly what we want for the $G$-action on $V_0$. 

So let 
$$\XX_0:=\bigcup_{v_i\in \hat{S}[k]}(V_i, D_i)$$
where we identify $D_{ij}$ with $D_{ji}$ to make the nodes of $D_i$ are identified with the nodes of $D_j$.  It is routine to check that the triple formula 
holds so that $\XX_0$ is a Type III anti-canonical pair and there exists a $G$-action on the Inoue-Hirzebruch surface $V_0$. 
The $G$ action on other $V_i$'s for $i>0$ depends on the action of $G$ on $V_0$.  We know that the $G$-action on $(V_0, D^\prime)$ is hyperbolic. 
For  the Looijenga pair $(V_i, D_i) (i>0)$ in $\XX_0$,  $D_i=\cup D_{ij}$, the $G$-action on a neighborhood $D_{i0}$ should be compatible with the action of $G$ on the neighborhood of $V_{D_0}=V_{D^\prime}\subset \XX_0$. 
The $G$-action on the whole  $(V_i, D_i)$  for $i>0$ should be compatible with the action above such that the quotient 
$\Gamma(\XX_0)/G$ is also a sphere. 
There are several modifications for the construction of P. Engel in \cite[\S 5.4, \S 5.5, \S 5.6]{Engel} that we do not need to discuss.


\section{The main result}

\subsection{The proof of Theorem \ref{thm_cusp_equivariant_smoothing}}

We prove Theorem  \ref{thm_cusp_equivariant_smoothing}.  Suppose  the surface cusp singularity $(\overline{W}, q^\prime)$ admits a smoothing such that it is 
induced from the  $G$-equivariant smoothing of 
the cusp $(\overline{V}, p^\prime)$.  We let 
$$\pi: \overline{\sV}\to \Delta$$
be the $G$-equivariant smoothing of 
the cusp $(\overline{V}, p^\prime)$.
Recall that in \S \ref{subsec_G_Inoue}, we construct the $G$-Inoue-Hirzebruch surface 
$(\bbV, p, p^\prime)$ and $(V, D, D^\prime)$. We may construct the smoothing $\pi: \overline{\sV}\to \Delta$ as follows: 
first we take the $G$-equivariant smoothing of $(\bbV,p^\prime)$,
$$\pi: \overline{\overline{\sV}}\to \Delta$$
such that each fiber $\overline{\overline{\sV}}_t$ contains a cusp singularity $p_t$ for $t\neq 0$. 
Therefore, we take the $G$-equivariant simultaneous resolution of singularities of  $\{p_t\}$ to obtain 
$$\pi: \overline{\sV}\to \Delta$$
such that the generic fiber $\overline{\sV}_t=(Y_t, D_t)$ is a rational surface $Y_t$, together with an anti-canonical divisor 
$D_t$.  This is a $G$-Looijenga pair.   From the construction in  \S \ref{subsec_G_Inoue} again,  the group 
$G$ acts on $Y_t\setminus D_t$ may has quotient singularities, and  $D_t$ is negative-definite which can be contracted to the cusp singularity $p_t$. 
The quotient $(Y_t\setminus D_t)/G$ is still an Inoue-Hirzebruch surface.  We can take resolution of singularities for 
$Y_t/G$ again such that the quotient is a smooth Looijenga pair $(X, E)$. 

Conversely,   suppose that there is a negative-definite Looijenga pair $(Y,D)$, together with a hyperbolic  finite group $G$-action such that after 
possible resolution of singularities the quotient $(Y, D)/G$ becomes a Looijenga pair $(X,E)$,  we need to show that the dual cusp $p^\prime$ of the cusp $D$ admits a $G$-equivariant smoothing which induces a smoothing 
of the dual quotient cusp $q^\prime$ of $q$ corresponding to $E$. 

From the construction in \S \ref{subsec_construction}, we have  the following  Type III canonical degeneration pairs from a 
$G$-Looijenga pair $(Y,D)$: 
$$\XX_0=\bigcup_{i\in I, i\ge 0}(V_i, D_i)$$
such that $(V_0, D, D^\prime)$ is an Inoue-Hirzebruch surface.  Also there is a $G$-action on the surface $V_0$ such that if 
$(V, D, D^\prime) \to (\bbV_0, p, p^\prime)$ is the contraction of $D, D^\prime$ to $p, p^\prime$, then the $G$-action on $\bbV_0$ only has 
two fixed points $p, p^\prime$. 

Let $\XX\to \Delta$ be the deformation of $\XX_0$. 
From \cite{Shepherd-Barron2}, all the components $\sum_{i\ge 1}(V_i, D_i)$ is contractible,  and we get 
\[
\xymatrix{
\XX\ar[rr]\ar[dr]_{\pi}&&\overline{\XX}\ar[dl]^{\overline{\pi}}\\
&\Delta&
}
\]
such that $\overline{\XX}\to \Delta$ is a deformation of $(\bV_0, D, p^\prime)$
which is a $G$-equivariant smoothing of the cusp $p^\prime$. 

Note that the fiber $(Y_t, D_t)$ of the $\overline{\pi}: \overline{\XX}\to \Delta$ for $t\neq 0$ admits a $G$-action such that the $G$-action is free on 
$Y_t-D_t$, and  $(Y_t, D_t)$ is negative-definite.  Then we simultaneously  contract such $D_t$'s and get 
$$\overline{\overline{\pi}}: \overline{\overline{\XX}}\to \Delta$$
which is a smoothing of $(\bbV_0, p, p^\prime)$. 

Then we take the quotient $\pi: \overline{\overline{\sX}}=\overline{\overline{\XX}}/G\to \Delta$ such that it is a smoothing of $\bbW_0=\bbV_0/G$, and this is another Inoue-Hirzebruch surface
$(\bbW_0, q, q^\prime)$.  We take simultaneously resolution of singularities 
for $q_t\in \overline{\overline{\sX}}_t$ and get 
$$\overline{\sX}\to \Delta$$
which is a smoothing of $(\bW_0, E, q^\prime)$ and the fiber $(X_t, E_t)$ is a Looijenga pair.  
From Proposition \ref{prop_corner_internal_blow-up_G}, the pair  $(X_t, E_t)$ is the quotient of the pair $(Y_t, D_t)$ for $t\neq 0$. 
This gives the smoothing of the cusp $q^\prime$.

\subsection{Example 1}\label{subsec_example1}

We consider an example.  Let  $(\overline{V}, p^\prime)$ be a negative definite cusp singularity whose resolution cycle (in terms of negative self-intersection numbers) is 
$\mathbf{d}^\prime=(5,2)$.  This is a hypersurface cusp given by 
$$\{x^3+y^3+z^5+xyz=0\}.$$
From \cite[Page 308, Example]{Pinkham},  this cusp admits a $G=\zz_2$-action whose quotient is a cusp   
$(\overline{W}, q^\prime)$  whose resolution cycle is given by 
$(2+6)$.   Note that the resolution cycle is a rational nodal curve $E^\prime$ with self-intersection number $-6$. It is a complete intersection singularity, see Example \ref{example_cusp_1} in \S \ref{subsec_G_Inoue}. 

The dual cusp $q$ of $(\overline{W}, q^\prime)$ is given by $\mathbf{d}=(3,2,2,2,2,2)$; and 
the dual cusp  of $(\overline{V}, p^\prime)$ is given by $(4,2,2)$.   Thus, we have the following diagram:
\[
\xymatrix{
(\overline{V}, p^\prime)\ar@{<->}[r]\ar[d]& (Y, D)\ar[d]\\
(\overline{W}, q^\prime)\ar@{<->}[r]& (X,E)
}
\]
where $(Y, D)$ is the hyperbolic Looijenga pair with negative self-intersection sequence $(4,2,2)$, and $(X,E)$ is the Looijenga pair with an anti-canonical divisor 
$E$ whose  negative self-intersection sequence is given by $\mathbf{d}=(3,2,2,2,2,2)$. Looijenga conjecture implies that both  $(\overline{W}, q^\prime)$ and  $(\overline{V}, p^\prime)$ admit smoothings. 

We need to find a $\mu_2$ action on the Looijenga pair $(Y,D)$.  The cusp associated with $(4,2,2)$ is a hypersurface cusp
$$\{x^2+y^4+z^7+xyz=0\}.$$
Let $\mu_2=\langle\eta\rangle$ act on this cusp as:
$$x\mapsto \eta x;\quad y\mapsto \eta y; \quad z\mapsto z.$$
Then we can calculate the invariant cusp which is the hypersurface cusp 
$$\{x^2+y^3+z^{12}+xyz=0\}$$
whose minimal resolution cycle is $(3,2,2,2,2,2)$ from \cite[Lemma 2.5]{Nakamura}.

From \cite[Theorem 1.1]{Looijenga},  since the length of $(4,2,2)$
is $3$, then we can blow down disjoint interior curves intersecting with $D_i$'s in $D$
to get the toric model $(Y^{\toric}, D^{\toric})=(\pp^2, \overline{D})$, where $\overline{D}$ contains three boundary lines with negative self-intersection sequence $(-1,-1,-1)$.

We do the same process by internal blow-downs and the negative self-intersection sequence becomes:
$$(3,2,2,2,2,2)\to  (1,1,1,1,1,1).$$
Here the Looijenga pair $(X,E)$ has Picard number $10+1=11$, see \cite{Nakamura}.  The above blowed down $7$ $(-1)$-curves, therefore, the pair with cycle $(1,1,1,1,1,1)$ has Picard number $4$.  

Then we get the toric model  $(X^{\toric}, E^{\toric})$, where $X^{\toric}$ is the blow-up $Bl_3\pp^2$, and  $E^{\toric}$ is the boundary divisors which is given by $(1,1,1,1,1,1)$.  This is a degree $6$ del Pezzo surface. 
We see that the $\zz_2$ acts on $(\pp^2, \overline{D})$ has the following form 
$$\eta[x:y:z]=[x:\eta y, \eta z].$$
So the fixed point locus are a point and a line.  The Quotient is still a $\pp^2$ with boundary divisors three intersecting lines. 
 Taking blow-ups along these three points we get  $(X^{\toric}, E^{\toric})$. 

\subsection{Example 2}\label{subsec_example2}

All of the cusps in \S \ref{subsec_example1} are lci cusps. We give an example of the group action of cusps which are not lci. 

We consider   $(\overline{W}, q^\prime)$ to  be a negative definite cusp singularity whose resolution cycle is 
$\mathbf{d}^\prime=(5,11, 2)$.  This is not an lci cusp since the dual cusp 
has resolution cycle $E$ whose self-intersection sequence is given by $(2,2,3\underbrace{2,\cdots, 2}_{8},4)$ with length $12$. This is a hypersurface cusp given by 
$$\{x^3+y^{12}+z^6+xyz=0\}.$$  
Already \cite[Proposition 4.8]{FM} showed that $E$ lies in a rational surface $X$ as an anti-canonical divisor. 
One can calculate  the torsion subgroup 
$$H_1(\Sigma, \zz)_{\tor}=\zz_3\oplus\zz_{30}$$ where $\Sigma$ is the link of the cusp $E$, see \cite[Example in \S 3]{Pinkham}. 
Pinkham constructed a 
$G=\zz_3$-cover of $X$ and the cusp $E$ by the result in Theorem \ref{thm_lattice_D_Lambda}. Let us denote the cover by $Y$.
Pinkham construced the cover first on the contraction of $E\mapsto q\in \overline{X}$, and then calculate the $\zz_3$-cover $\overline{Y}$ of $\overline{X}$ which contains another cusp whose resolution cycle $D$ is given  by 
$(4,3,2,3,2,2,2)$. From \cite[Lemma 2.5]{Nakamura}, this cusp is a complete intersection cusp whose local equation is given by 
$$x^2+w^4=yz, \quad  y^3+z^6=xw.$$

Thus, we get the Looijenga pair $(Y,D)$ endowed with a $\zz_3$-action with quotient $(X,E)$.
Thus, we have the following diagram:
\[
\xymatrix{
(\overline{V}, p^\prime)\ar@{<->}[r]\ar[d]& (Y, D)\ar[d]\\
(\overline{W}, q^\prime)\ar@{<->}[r]& (X,E).
}
\]

The dual cusp $(\overline{V},p^\prime)$ of the cusp $D$ has resolution cycle $(2,3,4,6)$, which is also not lci. From our main theorem, $\zz_3$ acts on the cusp $(\overline{V},p^\prime)$ with quotient the cusp $(\overline{W},q^\prime)$
Equivariant Looijenga conjecture implies that   $(\overline{V}, p^\prime)$ admits a $\zz_3$-equivariant smoothing which induces a smoothing  of $(\overline{W},q^\prime)$. 
We can calculate the monodromy matrices of the cusps $(\overline{V},p^\prime)$ and 
 $(\overline{W},q^\prime)$, and they have the same trace. Thus, \cite{Hirzebruch} implies that these two cusps are mutual fiberwise covers of each other.

\section{Equivariant smoothing of cusps by lci cusps}

\subsection{Smoothing by lci cusps}

Let $(V,p^\prime)$ be a cusp singularity with the resolution cycle $D$. 
We call it an $\lci$ cusp if as a singularity it is $\lci$. \cite[Theorem]{Karras} classified all the lci cusp singularities. 

\begin{thm}\label{thm_lci_cusps}(\cite{Karras})
    A cusp singularity $(V,p^\prime)$ is a local complete intersection ($\lci$) cusp if and only if $(V,p)$ is one of the following:
    \begin{enumerate}
        \item $T_{p,q,r}: x^p+y^q+z^r-xyz=0$ with $\frac{1}{p}+\frac{1}{q}+\frac{1}{r}<1$,
        \item $\prod_{p,q,r,s}: x^p+w^r=yz,\quad  y^q+z^s=xw$ with $(\frac{1}{p}+\frac{1}{r})(\frac{1}{q}+\frac{1}{s})<1$,
    \end{enumerate}
    where $p,q,r$ and $p,q,r,s$ are integers greater than $1$ and the point $p^\prime$ is chosen to be the origin. 
\end{thm}

The resolution cycles $D$ with negative self-intersection sequence $d=(d_0,\cdots, d_n)$ of the lci cusps are given in \cite[Lemma 2.5]{Nakamura}. Using the main result in the paper, we prove

\begin{thm}\label{thm_smoothing_lci_cusp}
  Let $(\overline{W}^\prime,q^\prime)$ be a cusp singularity.   Suppose that $(\overline{W}^\prime,q^\prime)$ admits a smoothing $f: \sW\to \Delta$.  Then there exists a smoothing $\sV\to \Delta$ of an lci cusp
  $(\overline{V}, p^\prime)$ together endowed with  a finite group $G$ action such that the quotient induces the smoothing $f: \sW\to \Delta$.
\end{thm}
\begin{proof}
    Let $(\overline{W}^\prime,q^\prime)$ be a cusp whose resolution cycle $E^\prime$ is given by 
    $(d^\prime_0,\cdots, d_k^\prime)$. Since  $(\overline{W}^\prime,q^\prime)$ admits a smoothing $f: \sW\to \Delta$, from Looijenga Conjecture/Theorem there is a Looijenga pair $(X,E)$ such that $E$ is an anti-canonical divisor with minimal resolution cycle $(d_0,\cdots, d_n)$ which is the dual of $(d^\prime_0,\cdots, d_k^\prime)$.
    If the length $n+1\le 4$, then the cusp $(\overline{W}^\prime,q^\prime)$ must be an $\lci$ cusp, and the smoothing $f: \sW\to \Delta$ is an lci smoothing.

    We suppose that $n+1>4$ so that $(\overline{W}^\prime,q^\prime)$ is not $\lci$. 
    The monodromy matrix of the cusp $E$ is given by 
    $$\sigma=
\mat{cc}
0&-1\\
1&d_n
\rix\cdots \mat{cc}
0&-1\\
1&d_0
\rix
=
\mat{cc}
a&b\\
c&d
\rix.
$$
We let $E$ contracts to get a cusp $(\overline{W},q)$ which is dual to $(\overline{W}^\prime, q^\prime)$. 
Neumann and Wahl, in \cite[Proposition 4.1 (2)]{NW},  constructed a finite  cover 
$(\overline{V}, p)$ of $\overline{W}$ with transformation group $G$ so that $(\overline{V}, p)$ is  a hypersurface cusp, which is an $\lci$ cusp. 
The construction is as follows. 
Let $H$ be the subspace of $\zz^2$ generated by $\mat{c}a\\ c\rix$ and $\mat{c}0\\ 1\rix$.   We can assume $a\neq 0$, otherwise we just take $H=\zz^2$.   Then the  matrix $\sigma$ takes the subspace $H$ to itself
by the matrix  $\mat{cc}0&-1\\ 1& t\rix$ where $t=\tr(\sigma)=a+d$.
The finite transformation group $G$ is  given as follows:   first we take  the quotient finite group $N(H\rtimes \zz)/H\rtimes \zz$, where $N(H\rtimes \zz)$ is the normalizer.  Then  the subgroup 
$H\rtimes \zz\subset \pi_1(\Sigma)$ determines a cover  of $\overline{V}$. 
This cover determined by $H\rtimes \zz$ is either the cusp with resolution graph consisting of a cycle with one vertex weighted $-t$ or the dual cusp of this, according as the above basis is oriented correctly or not, i.e., whether $a< 0$ or  $a > 0$. We list these two cases:
\begin{enumerate}
    \item If $a<0$, then we just use the cusp $(\overline{V},p)$ whose resolution cycle is $(t)$. There is a finite cover $(\overline{V},p)\to (\overline{W},q)$.  
    \item If $a>0$, the cover $(\overline{V},p)$ is the dual cusp of $(t)$, hence must be a hypersurface cusp:
$$(3, \underbrace{2,2,\cdots, 2}_{t-3})$$
with local equation given by 
$$\{x^2+y^{3}+z^{t-3+7}+xyz=0\}.$$
From \cite[Theorem]{Pinkham}, the torsion subgroup of $H_1(\widetilde{\Sigma},\zz)$ is 
$$B=\{(\lambda,\mu,\nu)\in\cc| \lambda^2=\mu^3=\nu^{t+4}=\lambda \mu\nu\}$$
and this group $B$ acts on the above cusp gives the dual cusp $(t)$. 

Also \cite{Hirzebruch} implies that if the cusp $(\overline{V},p)$ and the cusp $(t)$ are dual to each other, then the cusp $(t)$ is also a finite cover over the cusp  $(\overline{V},p)$. 
We have finite covers $(t)\to (\overline{V},p)\to (\overline{W},q)$.
Thus, there is still a finite cover of the cusp $(t)$ to $(\overline{W},q)$. 
We still denote the cusp $(t)$ as $(\overline{V},p)$.
\end{enumerate}

In summary, we get a finite cover $(\overline{V},p)\to (\overline{W},q)$ with transformation group $G$ such that 
$(\overline{V},p)$ has resolution cycle $(t)$. 
Now we argue that $(\overline{V},p)$ can be fitted into a Looijenga pair $(Y,D)$ with minimal resolution cycle  $(t)$. 
We work on analytic spaces for the cusp $(\overline{V},p)$, and the Looijenga pair $(X,E)$.  The contraction of $E$ gives the cusp $(\overline{W}, q)$, and $G$ is a subgroup of the fundamental group of the link of the cusp$(\overline{W}, q)$ since there is a $G$-cover $(\overline{V},p)\to (\overline{W},q)$.  

Let $U_E$ be a neighborhood of $E$ in $X$.   Then the contraction of $E$ to the cusp $q$ gives a neighborhood  $\overline{U}_q\subset \overline{W}$ of $q$.  Let $\Sigma=\partial U_E$ is the boundary, which is also the boundary of $\overline{U}_q$. Then $\Sigma$ is the link of the cusp $(\overline{W}, q)$. 
Shrink $U_E$ to a tubular neighborhood, then $\overline{U}_q\setminus \{q\}$ is homotopy equivalent to the link.  Thus, we have $X=(X\setminus E)\cup V_E$. 
Let $\overline{X}=(X\setminus E)\cup \overline{U}_q$. 
From Van Kampen theorem  we can glue the local singularity germ  $(\overline{V},p)$ (a neighborhood) to the suitable $G$-cover of $X\setminus E\cong \overline{X}\setminus \{q\}$ and then take resolution to get $(Y, D)$. 
Then the group $G$ acts on $(Y, D)$ which is hyperbolic, and from the quotient $(Y,D)/G$, after performing  suitable resolution of singularities we get $(X,E)$.

Therefore, from main Theorem \ref{thm_cusp_equivariant_smoothing}, and also the dual cusp of $D$ is a hypersurface cusp $(\overline{V}^\prime, p^\prime)$ whose equivariant smoothing induces a smoothing of the cusp $(\overline{W}^\prime, q^\prime)$.
\end{proof}

\subsection{Example 3}\label{subsec_example3}

Here is another example.  Let  $(\overline{W}^\prime, q^\prime)$ be a negative definite cusp singularity whose resolution cycle (in terms of negative self-intersection numbers)  is 
$(9,6)$, which is not a complete intersection cusp since 
$(6-2)+(9-2)=11>4$.  Let $E^\prime$ be the resolution cycle.  The charge $Q(E^\prime)=12+(6-3)+(9-3)=21$. Then from Proposition \ref{prop_rational_dual},  this cusp 
$E^\prime$ has a rational dual $E$ with negative self-intersection sequence $(3,2,2,2,3,2,2,2,2,2,2)$. 

This cusp $(3,2,2,2,3,2,2,2,2,2,2)$ is a hypersurface cusp. From \cite[Lemma 2.5]{Nakamura}, the local  equation is given by 
$$x^2+y^8+z^{11}-xyz=0.$$
Also let $\Sigma$ be its link, then we have 
$$H_1(\Sigma,\zz)_{\tor}=G=\{\lambda, \mu,\nu\in\cc|\lambda^2=\mu^8=\nu^{11}=\lambda \mu\nu\}.$$
Its dual cusp has resolution cycle $(1, 7, 10)$ and minimal resolution cycle $(9,6)$. 
\cite[Theorem]{Pinkham} implies that the $G$ action on the cusp $\mathbf{d}=(3,2,2,2,3,2,2,2,2,2,2)$ (which we denote it by $(\overline{V}^\prime, p^\prime)$) is given by 
$$x\mapsto \lambda x, \quad y\mapsto \mu y, \quad z\mapsto \nu z.$$
The quotient $(\overline{V}^\prime, p^\prime)/G$ is the cusp $(\overline{W}^\prime, q^\prime)$. 
Let us form the following diagram:
\[
\xymatrix{
(\overline{V}^\prime, p^\prime)\ar@{<->}[r]\ar[d]& (Y, D)\ar[d]\\
(\overline{W}^\prime, q^\prime)\ar@{<->}[r]& (X,E)
}
\]
where we use the  Looijenga conjecture (which is now a theorem).  $(\overline{V}^\prime, p^\prime)$ is a hypersurface cusp which admits a smoothing, hence there is a Looijenga pair $(Y,D)$ so that $D$ has minimal resolution cycle $(9,6)$ which is the dual of $(3,2,2,2,3,2,2,2,2,2,2)$. 

We also know the cycle $(3,2,2,2,3,2,2,2,2,2,2)$ is an anti-canonical divisor of a smooth rational surface $(X,E)$. 
Here is another way to see this.  
Blow-up $(\pp^1\times \pp^1, E=4 \text{~lines})$ along $4$ corners first and we get 
$(Bl_4\pp^1\times\pp^1)$ with $E$ becomes $(2,1,2,1,2,1,2,1)$.  Then we do blow-up along another three corners, and interior blow-ups along the components with self-intersection number $(-1)$, and blow down one interior $(-1)$-curve along one component 
with with self-intersection number $(-3)$. Finally we get $(X,E)$. Thus, from Looijenga conjecture again the cusp $(\overline{W}^\prime, q^\prime)$ is smoothable. 

There should have a $G$-action on $(Y,D)$ such that its quotient $(Y,D)/G$, after suitable resolution of singularities, we get $(X,E)$. Since the length of $D$ is $2$, from \cite[Theorem 1.1]{Looijenga}, after blowing down the disjoint $(-1)$-curves we get the toric model $(\pp^1\times \pp^1, \overline{D})$, where $\overline{D}$ contains two components $\overline{D}_0, \overline{D}_1$ with bidegree $(1,1)$.  
We can perform toric blow-ups to get the toric model $(Y^{\toric}, D^{\toric})$ such that 
$D^{\toric}$ contains $8$ $\pp^1$ components with negative self-intersection sequence $(2,1,2,1,2,1,2,1)$.  

On the other hand, starting from the cycle $E$, since $Q(X,E)=12-9=3$ we can perform internal blow-downs  and corner-blow-downs to get 
$$
(3,2,2,2,3,2,2,2,2,2,2)\to (3,1,2,2,3,1,2,2,2,1,2)\to (2,1,2,2,1,2,1,1).
$$
Thus, we have the toric model $(X^{\toric}, E^{\toric})$ with cycle $(2,1,2,2,1,2,1,1)$
since its charge is zero. We can perturb some components to get $(2,1,2,1,2,1,2,1)$.

\subsection{Example 4}\label{subsec_example4}
We still look at the cusp $(\overline{W}^\prime, q^\prime)$ as in Example in \S \ref{subsec_example3}. 
Its dual cusp $(\overline{W},q)$ has resolution cycle  $E$ with negative self-intersection sequence $(3,2,2,2,3,2,2,2,2,2,2)$. 
We have its monodromy matrix 
$$\sigma^\prime=\mat{cc}
0&-1\\
1&2
\rix\cdot\mat{cc}
0&-1\\
1&2
\rix\cdots  \mat{cc}
0&-1\\
1&3
\rix
=
\mat{cc}
-34&-75\\
39&86
\rix.$$

From \cite[Proof of Theorem 4.1]{NW},  $t=86-34=52$, there exists a  finite cover 
$(\overline{V}, p)\to (\overline{W}, q)$ such that $(\overline{V}, p)$ is a cusp with 
resolution cycle $(52)$. 
The dual cusp  $(\overline{V}^\prime, p^\prime)$ to the cusp  $(\overline{V}, p)$
is a 
hypersurface cusp whose resolution cycle is given by 
$(3, \underbrace{2, 2, \cdots, 2}_{49})$. 
From \cite[Lemma 2.5]{Nakamura}, this hypersurface  cusp  is given by 
$$\{x^2+y^3+z^{56}+xyz=0\}.$$
Let $G$ be the transformation group of the cover $(\overline{V}, p)\to (\overline{W}, q)$.  First we have the exact sequence
$$0\to H\rtimes \zz\to \pi_1(\Sigma)\to G^\prime\to 0$$
where $H\subset \pi_1(\Sigma)=\zz^2\rtimes \zz$ is generated by 
$(0,1), (-34,39)$. 
Let $G^\prime=N(H)/H$, where $N(H)$ is the normalizer. 
Then 
$$G=G^\prime\rtimes H_1(\Sigma,\zz)_{\tor}.$$ 

Note that the cusp singularity $(\overline{V}, p)$ has monodromy 
$$\sigma=\mat{cc}
0&-1\\
1& 52
\rix.$$
This matrix has the same trace with $\sigma^\prime$. Therefore, from \cite{Hirzebruch}, these two cusps are mutual fiberwise covering spaces.
Here we use the cover $(\overline{V}, p)\to (\overline{W}, q)$.

Let $D$ be the resolution cycle representing $(52)$, and $E$ be the resolution cycle representing $(3,2,2,2,3,2,2,2,2,2,2)$.  Then these two cycles lie in a Looijenga pairs 
$(Y,D)$ and $(X,E)$. 
Thus, we have the following diagram:
\[
\xymatrix{
(\overline{V}^\prime, p^\prime)\ar@{<->}[r]\ar[d]& (Y, D)\ar[d]\\
(\overline{W}^\prime, q^\prime)\ar@{<->}[r]& (X,E).
}
\] 

From \cite[Figure 12]{Engel},  the Type III anti-canonical pairs of $(9,6)$ and $(3,2,2,2,3,2,2,2,2,2,2)$ are given.  
Similar way we can get the Type III anti-canonical pairs of $(3, \underbrace{2, 2, \cdots, 2}_{49})$ and  $(52)$, whose quotient under $G$ gives the above Type III anti-canonical pairs, up to 
resolution of singularities. We check that the quotient $(Y,D)/G$, after corner blow-ups and blow-downs, internal blow-ups and blow-downs we can get $(X,E)$. From \cite[Theorem 1.1]{Looijenga}, the pair $(Y,D)$, after blowing down disjoint exceptional curves, becomes $(\pp^2, \overline{D})$, where $\overline{D}$ is a cubic curve with a node. We perform $10$ times corner blow ups and get the cycle $(2,2,2,2,2,2,2,2,2,1,3)$.  Then we do two internal blow-ups again to get $(3,2,2,2,3,2,2,2,2,2,2)$.



\subsection*{}

\end{document}